\documentclass[a4paper,12pt,reqno]{amsart}

\usepackage{amsthm}
\usepackage{amsmath}
\usepackage{amssymb}
\usepackage{mathrsfs}
\usepackage{latexsym}
\usepackage{exscale}
\usepackage{geometry} % margenes
\usepackage{graphicx} % graficos
\usepackage[utf8]{inputenc}
\usepackage{verbatim}
\usepackage{enumerate} % enumerar con letras
\usepackage{ esint } %promedio, integral con raya
\usepackage{fancyhdr}

\usepackage{color}
\definecolor{red}{rgb}{1,0.1,0.1}%valores de las componentes roja, verde y azul
\definecolor{blue}{rgb}{0.1,0.1,1}
\definecolor{vb}{RGB}{160,32,240}
\newcommand{\red}{\color{red}} % Color del texto a partir de ese punto

\usepackage[colorlinks=true, pdfstartview=FitV, linkcolor=blue, citecolor=red, urlcolor=blue]{hyperref} %%% % Includes %hyperrefs into pdf

\theoremstyle{plain}

\newtheorem*{teo*}{Theorem}
\newtheorem*{prop*}{Proposition}
\newtheorem*{lema*}{Lemma}

\numberwithin{equation}{section}
\newtheorem{teo}{Theorem}[section]
\newtheorem{lema}[teo]{Lemma}
\newtheorem{corol}[teo]{Corollary}
\newtheorem{prop}[teo]{Proposition}

\theoremstyle{remark}
\newtheorem{remark}[teo]{Remark}

\theoremstyle{definition}

\newtheorem*{mydef*}{Definition}
\newtheorem{mydef}[teo]{Definition}
\newtheorem{ejem}[teo]{Example}

\calclayout

\allowdisplaybreaks

\newcommand{\R}{\mathbb{R}^n}
\newcommand{\D}{\mathcal{D}}
\begin{document}

\title[Necessary condition on the weight for  maximal and integral operators]{Necessary condition on the weight  for  maximal and integral operators with rough kernels}

\author[G.~H.~Iba\~{n}ez~Firnkorn]{Gonzalo H. Iba\~{n}ez-Firnkorn}
\address{G.~H.~Iba\~{n}ez~Firnkorn\\ FaMAF \\ Universidad Nacional de C\'ordoba \\
CIEM (CONICET) \\ 5000 C\'ordoba, Argentina}
\email{gibanez@famaf.unc.edu.ar}

\author[M.~S.~Riveros]{Mar\'{\i}a Silvina Riveros}
\address{M.~S.~Riveros \\ FaMAF \\ Universidad Nacional de C\'ordoba \\
CIEM (CONICET) \\ 5000 C\'ordoba, Argentina}
\email{sriveros@famaf.unc.edu.ar}

\author[R.~E.~Vidal]{Ra\'ul E. Vidal}
\address{R.~E.~Vidal \\ FaMAF \\ Universidad Nacional de C\'ordoba \\
	CIEM (CONICET) \\ 5000 C\'ordoba, Argentina}
\email{vidal@famaf.unc.edu.ar}

\thanks{ The authors are  partially supported by
CONICET and SECYT-UNC}



%\date{July 16, 2007. \textit{Revised}: \today}

\begin{abstract} Let $0\leq \alpha<n$, $m\in \mathbb{N}$ and let consider $T_{\alpha,m}$ be  a  
of integral operator, given by
kernel of the form
$$K(x,y)=k_1(x-A_1y)k_2(x-A_2y)\dots k_m(x-A_my),$$
where $A_i$ are invertible matrices and each $k_i$ satisfies a
fractional size  and generalized fractional H\"ormander condition.
In \cite{IFRi18}  it was proved that $T_{\alpha,m}$ is controlled in
$L^p(w)$-norms, $w\in A_{\infty}$, by the sum of
maximal operators $M_{A_i^{-1},\alpha}$. In this paper we present
the class of weights  $\mathcal{A}_{A,p,q}$, where $A$ is an
invertible matrix. This class are the good weights for the weak-type
estimate of $M_{A^{-1},\alpha}$. For certain kernels $k_i$
 we can characterize the weights for the strong-type
estimate of $T_{\alpha,m}$. Also, we give a  the strong-type
estimate using testing conditions.
\end{abstract}

\maketitle

\section{Introduction and  results}

In this paper we will characterized the good weights, $0\leq w\in L^1_{loc}(\R)$, for integral operators of the form
\begin{equation}
T_{\alpha,m}f(x)=\int_{\mathbb{R}^n}K(x,y)f(y)dy, \label{eq: defT}
\end{equation}
 where  the kernel
\begin{equation*} %\label{eq: defK}
K(x,y)=k_1(x-A_1y)k_2(x-A_2y)\dots k_m(x-A_my), \qquad m \in\mathbb{N},
\end{equation*}
  $A_i$ are certain invertible matrices
 and  $f \in L_{\text{loc}}^{\infty}(\mathbb{R}^n)$.
In \cite{RS88} the authors studied the $L^2(\mathbb{R})\to
L^2(\mathbb{R})$ boundedness for the case
$$K(x,y)=\frac 1{|x-y|^\alpha}\frac 1{|x+y|^{1-\alpha}}, \qquad\qquad 0<\alpha< 1.$$
In several articles, different authors studied the  boundedness in Lebesgue spaces and weighted Lebesgue spaces where the kernels $k_i$, $i=1,\dots,m $, satisfy certain integral H\"{o}rmander and size condition, for example see \cite{GU93}, \cite{IFRi18}, \cite{RU14}, \cite{RoU13}.

Let $0\leq \alpha<n$ and  $0<\alpha_i<n$, $1\leq i\leq m$,  such that  $\alpha_1 + \dots +\alpha_m=n-\alpha$. Also let $A_i$ be matrices such that

\

(H) $A_i$  and $A_i - A_j$ are  invertible for $1\leq i,j \leq m$ and $i\neq j$.

\

\noindent In the case where
$$k_i(x,y)=\frac 1{|x-A_iy|^{\alpha_i}}, $$
the  integral operator $T_{\alpha,m}$ satisfies the following Coifman--Fefferman inequality
\begin{equation}\label{cofe}
\int_{\mathbb{R}^n}|T_{\alpha,m}(f)(x)|^qw^q(x)\,dx\leq C_{w,q} \sum_{i=1}^m\int_{\mathbb{R}^n}| M_{\alpha ,A_i^{-1}}f(x)|^qw^q(x)\,dx,
\end{equation}
for all $0<q<\infty$,   $w^q$ a weight in the $\mathcal{A}_\infty$ Muckenhoupt class and where  $M_{\alpha,A^{-1}}$ is the   maximal  operator define by
\begin{equation}\label{defMA}
 M_{\alpha ,A^{-1}}f(x)=M_{\alpha }f(A^{-1}x)=\sup\limits_{Q\ni A^{-1}x }\frac{1}{%
\left\vert Q\right\vert ^{1-\frac{\alpha }{n}}}\int_{Q}f(y)dy.
\end{equation}
By a change of variable
\begin{equation}\label{acotMalphaA}
\int_{\mathbb{R}^n}| M_{\alpha ,A^{-1}}f(x)|^qw^q(x)\,dx%=\int_{\mathbb{R}^n}| M_{\alpha }f(A^{-1}x)|^qw^q(x)\,dx
=|\text{det}A| \int_{\mathbb{R}^n}| M_{\alpha }f(x)|^qw^q(Ax)\,dx.
\end{equation}
Now let   $1\leq p \leq q<\infty$, and $w$  in $\mathcal{A}_{ p,q}$, i.e.
\begin{align*}\label{pq}
[w]_{\mathcal{A}_{p,q}}:=&\underset{Q}{\sup} \left( \frac{1}{\left\vert Q\right\vert }\int\limits_{Q}w^{q}(x)dx\right) ^{%
\frac{1}{q}}\left( \frac{1}{\left\vert Q\right\vert }\int\limits_{Q}w^{-p^{%
\prime }}(x)dx\right) ^{\frac{1}{p^{\prime }}}\leq \infty, \quad &\text{ for } 1<p,\\
[w]_{\mathcal{A}_{1,q}}:=&\underset{Q}{\sup} \left( \frac{1}{\left\vert Q\right\vert }\int\limits_{Q}w^{q}(x)dx\right) ^{%
    \frac{1}{q}}\|w^{-1}\|_{\infty, Q}\leq \infty, & \text{ for } 1=p.
\end{align*}We say that $w^p\in \mathcal{A}_p$ if, and only if,  $w\in \mathcal{A}_{p,p}$, $1\leq p <\infty$, and $\mathcal{A}_\infty=\cup_p \mathcal{A}_p$.

For $p>1$, if we  request that $w(Ax)\leq Cw(x)$ and $w\in A_{p,q}$, then  in \eqref{acotMalphaA} we get
\begin{equation}\label{fuertepqMalphaA}
\int_{\mathbb{R}^n}| M_{\alpha ,A^{-1}}f(x)|^qw^q(x)\,dx \leq C_{A,q,p,w }\int_{\mathbb{R}^n}|f|^p w^p(x) \, dx.
\end{equation}
Finally in the equation \eqref{cofe} if $w\in \mathcal{A}_{p,q}$ and $w_{A_i}(x):= w(A_ix)\leq Cw(x)$ for all $1\leq i \leq m$, we have (see \cite{RU14})
\begin{equation}\label{fuerteTalphapq}
\int_{\mathbb{R}^n}|T_{\alpha,m}(f)(x)|^qw^q(x)\,dx\leq C_{A,q,p,w }\int_{\mathbb{R}^n}|f|^p w^p(x) \, dx.
\end{equation}
%In the similar way for $\alpha=0$, $w^p\in \mathcal{A}_p$, (this is  $w\in \mathcal{A}_{p,p}$), and $w(A_ix)\leq Cw(x)$ for all $1\leq i \leq m$ we get
%\begin{equation}\label{fuerteTalphap}
%\int_{\mathbb{R}^n}|T_{m}(f)(x)|^pw^p(x)\,dx\leq C_{A,p,w }\int_{\mathbb{R}^n}|f|^p w^p(x) \, dx.
%\end{equation}

In this paper we answer some of these questions: Is there  a characterization of  the   weights for the boundedness of  $M_{\alpha,A^{-1}}$ defined in \eqref{defMA}? If this characterization is obtain, are we able to present some weighted bounds as in  \eqref{fuerteTalphapq}?
For inequality  \eqref{fuerteTalphapq} $w^q\in \mathcal A_\infty$ is required. In some cases,  this requirement,  can be avoid?

To  study some of  these problems we define the next classes of weights. Let $A$ be an invertible matrix and $1\leq p \leq q < \infty$. % $1\leq p\leq \infty$, $1\leq q<\infty$.

A weight $w$ is in the class $\mathcal{A}_{A, p,q}$, if
\begin{align}\label{Apq}
[w]_{\mathcal{A}_{A,p,q}}:=&\underset{Q}{\sup} \left( \frac{1}{\left\vert Q\right\vert }\int\limits_{Q}w^{q}_A(x)dx\right) ^{%
\frac{1}{q}}\left( \frac{1}{\left\vert Q\right\vert }\int\limits_{Q}w^{-p^{%
\prime }}(x)dx\right) ^{\frac{1}{p^{\prime }}}\leq \infty,  \quad &\text{ for } p>1,\\
[w]_{\mathcal{A}_{A,1,q}}:=&\underset{Q}{\sup} \left( \frac{1}{\left\vert Q\right\vert }\int\limits_{Q}w^{q}_A(x)dx\right) ^{%
    \frac{1}{q}}\|w^{-1}\|_{\infty, Q}\leq \infty,  &\text{ for } p=1. \nonumber
\end{align}

 We obtain the following weak type $(p,q)$ characterization:
\begin{teo}\label{weakMalpha} Let $0\leq \alpha <n$, $1 \leq p< n/\alpha$ and $\frac1{q}=\frac1{p}-\frac{\alpha}{n}$. The
 maximal operator $M_{\alpha ,\,A^{-1}}$  is bounded from $L^p(w^p)$ into\label{key} $L^{q,\infty}(w^q)$ if, and only if, $w\in \mathcal{A}_{A, p,q}$. % is of weak type $p-q$ with respect to the weights $w^{p}-$ $w^{q}$  $
%If the fractional operator $M_{\alpha ,\,A^{-1}}$ is of weak type$p-q$ with respect to the weights $w^{p}-$ $w^{q}$ then $w\in \mathcal{A}_{A, p,q}. $
\end{teo}

This result follows in the same way as the classical one in \cite{MW74} or \cite{M72} taking into account that
    $$ w^q\{ x:M_{\alpha,A^{-1}}f(x)>\lambda\}=w^q_A\{ x:M_{\alpha}f(x)>\lambda\}.$$

% we get is a two weight  week inequality for the operator   $M_{\alpha,A^{-1}}$.
%\begin{proof}[Proof of Theorem \ref{weakMalpha}] Observe that given a weight $w$
%   $$ w^q\{ x:M_{\alpha,A^{-1}}f(x)>\lambda\}=w^q_A\{ x:M_{\alpha}f(x)>\lambda\},$$
%   then the theorem follows in the same way as the classical one in \cite{MW74} or \cite{M72}.
%\end{proof}

 In \cite{S82} E. Sawyer introduced the following definition:
\begin{mydef}
Let $0\leq \alpha < n$, $1<p\leq q <\infty$, let  $(u,v)$ be a pair of weights. The pair $(u,v)\in\mathcal{M}_{\alpha, p,q}$ if  satisfy the testing condition
$$[u,v]_{\mathcal{M}_{\alpha, p,q}}=\underset{Q}{\sup} \;\; v(Q)^{-1/p}\left(\int_Q M_{\alpha}(\chi_Q v)^{q}u\right)^{1/q}<\infty.$$
    \end{mydef}
For the classical   maximal operator $M_\alpha$    it is known the following two weight inequa\-lity:
\begin{teo}\cite{S82}
Let $0\leq \alpha < n$, $1<p\leq q <\infty$, let  $(u,v)$, be a pair of weights. The following statements  are equivalent:
%{\red este $\sigma$ es un peso cualquiera?  o es el $\sigma$ de otro peso?}
\begin{enumerate}[(i)]
\item $(u,v)\in \mathcal{M}_{\alpha, p,q}$
\item For every $f\in L^p(v)$,
$$\left(\int_{\mathbb{R}^{n}}M_{\alpha}(fv)^qu\right)^{1/q}\leq C_{n,p,\alpha}[u,v]_{\mathcal{M}_{\alpha, p,q}}\left(\int_{\mathbb{R}^{n}}|f|^pv\right)^{1/p}.$$
\end{enumerate}
\end{teo}

%Observe that, if we take $u=w^q_A$ and $v=w^{-p'}$, and if  $(w^q_A, %w^{-p'})$  satisfies the testing condition (this is $ %\mathcal{M}_{\alpha,A,w}:= \mathcal{M}_{\alpha,w^q_A,w^{-p'}}<\infty$)   %then for  $g=fv=f w^{-p'}$ we get
%$$\left(\int_{\mathbb{R}^{n}}M_{\alpha,A^{-1}}(g)^qw^q\right)^{1/q}\leq %C_{n,p,\alpha}\mathcal{M}_{\alpha, A, w}\left(\int_{\mathbb{R}^{n}}|g|^p %w^p\right)^{1/p}.$$
 \begin{mydef} Let $w$ be a weight, we say $w\in \mathcal{M}_{\alpha,A, p,q}$  if $(w^q_A, w^{-p'})\in \mathcal{M}_{\alpha,p,q}$ and $w\in \mathcal{M}_{\alpha,A, q',p'}$  if $(w^{-p'},w^q_A )\in \mathcal{M}_{\alpha,q',p'}$. For $\alpha=0$ and $p=q$, we say $w\in \mathcal{M}_{A, p}:=\mathcal{M}_{0,A, p,p}$.
 \end{mydef}
\begin{remark}\label{RemtestimplicaAApq}
    Also if $w\in \mathcal{M}_{\alpha,A, p,q}$ and  $\frac1p-\frac1q=\frac{\alpha}{ n}$, we have that  $w \in \mathcal{A}_{A,p,q}$ and $[w]_{\mathcal{A}_{A,p,q}}\leq [w]_{\mathcal{M}_{\alpha,A, p,q}}$.
\end{remark}

Then as a corollary of the previous theorem we have,
\begin{corol}\label{CaracPesosMA} Let $0\leq \alpha<n$,  $1<p<\frac n\alpha$  and $\frac 1q=\frac 1p-\frac \alpha n$.
The weight $w\in \mathcal{M}_{\alpha,A, p,q}$ if, and only if,
$$\left(\int_{\mathbb{R}^{n}}M_{\alpha,A^{-1}}(g)^qw^q\right)^{1/q}\leq C_{n,p,\alpha}[w]_{\mathcal{M}_{\alpha,A, p,q}}\left(\int_{\mathbb{R}^{n}}|g|^p w^p\right)^{1/p},$$
  for  $g=fv=f w^{-p'}\in L^p(w^p)$, and where $[w]_{\mathcal{M}_{\alpha,A, p,q}}:=[w^q_A, w^{-p'}]_{\mathcal{M}_{\alpha, p,q}}.$
\end{corol}

%bserve now if $\alpha =0$, and $p=q$ we can write the previous theorem and remark as follows
%\begin{corol}
%Let $A$ be an invertible matrix and let $w$ be  a weight. The following are equivalent
%\begin{enumerate}[(i)]
%\item $w\in \mathcal{M}_{A, p}$, this is:
%$$\underset{Q}{\sup\;} \frac1{w^{-p'}(Q)}\int_Q M_{A^{-1}}(w^{-p'} \chi_Q)^pw^p<\infty.$$
%\item The operator $M_{A^{-1}}$ is bounded from
%$:L^p(w^p)$ into $L^{p}(w^p)$
%\end{enumerate}
%\end{corol}

In this paper we will  prove $L^p(w^p)$ into $L^q(w^q)$ bounds for
 integral operators defined in \eqref{eq: defT}, with kernels
satisfying a  fractional size,  $S_{\alpha,r}$, and a fractional
H\"ormander, $H_{\alpha,r}$, condition (for definitions see Section
2), without using  $w\in \mathcal{A}_\infty$. To obtain these
results we will use  the sparse domination technique. In the last
years this technique was used to obtain    sharp weighted norm
inequalities for singular or fractional integrals operators, for
example see \cite{L16}  and  \cite{AMPRR17}.

In the case for the  integral operator $T_{\alpha,m}$ with some particular matrices $A_i$ we obtain a norm estimate relative to  the constant of the weight.

\begin{teo}\label{APart}
    Let $0\leq \alpha < n$, $m \in \mathbb{N}$ and let  $T_{\alpha,m}$
    be the integral operator defined by (\ref{eq: defT}). For $1\leq i
    \leq m$, let $1<r_i\leq\infty$ and $0\leq\alpha_i<n$
    such that $\alpha_1+\cdots + \alpha_m=n-\alpha$. Let $k_i \in
    S_{n-\alpha_i, r_i}\cap H_{n-\alpha_i,r_i}$ and let the
    matrices $A_i$ satisfy  the hypothesis $(H)$.

    \noindent If $\alpha=0$, suppose $T_{0,m}$ be of strong type $(p_0,p_0)$ for some $1<p_0<\infty$.

    If there exists $s\geq 1$ such that $\frac1{r_1}+\cdots +\frac1{r_m}+\frac1{s}=1$,  $A_j=A_i^{-1}$ for some $j\not=i$, $s<p< \frac{n}{\alpha}$, $\frac1{q}=\frac1{p}-\frac{\alpha}{n}$ and $w^s\in \bigcap_{i=1}^m\mathcal{A}_{A_i,\frac{p}{s},\frac{q}{s}}$ for all $1\leq i \leq m$, then

    $$
    \|T_{\alpha,m}f\|_{L^q(w^q)}\leq C \|f\|_{L^p(w^p)}\sum_{i=1}^m[w^s]_{\mathcal{A}_{A_i,\frac{p}{s},\frac{q}{s}}}^{\max\left\{1-\frac{\alpha}{n},\frac{(p/s)'}{q}\left(1-\frac{\alpha s}{n}\right)\right\}}.$$

\end{teo}

    \begin{remark}
        Observe that if $A_j=A_i^{-1}$ for some $j\not=i$ and $w^s\in \mathcal{A}_{A_i,\frac{p}{s},\frac{q}{s}}\cap \mathcal{A}_{A_j,\frac{p}{s},\frac{q}{s}}$,
        then  $w^s\in \mathcal{A}_{\frac{p}{s},\frac{q}{s}}$ and $w_{A_j}\simeq w$. Then, in this case,  $w^q$ and $w^{-s(p/s)'}$ belongs to  $A_{\infty}$.
    \end{remark}

 In the case of the  integral operator
\begin{equation}\label{Talpha2}
 T_{\alpha, 2}f(x)=\int \frac{f(y)}{|x-A_1y|^{\alpha_1}|x-A_2y|^{\alpha_2 }}dy,
 \end{equation}
we have the following characterization.

\begin{teo}\label{Ejem}
Let $0\leq \alpha<n$, $0\leq \alpha_1,\alpha_2<n$ such that $\alpha_1+\alpha_2=n-\alpha$. Let $1<p<n/\alpha$ and $\frac1{q}=\frac1{p}-\frac{\alpha}{n}$. Let $A_1,A_2$ be a invertible matrices such that $A_1-A_2$ is invertible.
Let $T_{\alpha, 2}$ be the  integral operator defined by \eqref{Talpha2}.
Let $w$ be a weight. If $T_{\alpha, 2}:L^p(w^p)\to L^q(w^q)$ then $w\in \mathcal{A}_{A_1,p,q} \cap \mathcal{A}_{A_2,p,q}$.\\
Furthermore, if $A_2=A_1^{-1}$ or $A_1=-I$ and $A_2=I$, then $T_{\alpha, 2}:L^p(w^p)\to L^q(w^q)$ if and only if $w\in \mathcal{A}_{A_1,p,q} \cap \mathcal{A}_{A_2,p,q}$.

\end{teo}

\begin{remark}
This result contains   Theorem 3.2 and Corollary 3.3  in \cite{FF15}, where the authors consider $p=q$, $\alpha =0$,  $w^p(x)=|x|^\beta\in A_p$, $A_1=-I$ and $A_2=I$.
\end{remark}

%{\red
%\begin{remark}
%Observe that $k_i(x)=|x|^{-\alpha_i}$ then $k_i\in S_{n-\alpha_i,\infty}\cap  H_{n-\alpha_i,\infty}$. Then the operator defined in \ref{Talpha2} is a example of our operators.
%\end{remark}
%}

\begin{remark}
As is well known the  $\mathcal{A}_p$ condition, for $1< p\leq \infty$,  is also necessary for  Calder\'{o}n-Zygmund  operators,  in the following way: if  $w^p$ is a weight such that all the Riezs transforms are strong $(p,p)$, then $w^p\in \mathcal {A}_p$. In a similar way we see in this paper that the  $\mathcal{A}_{A,p,q}$ classes of weights are also necessary to obtain strong $(p, q)$ bounds for certain integral operators.
\end{remark}

The paper continuous in the following way: in Section 2 we give
preliminaries,  definitions and  we state some of our main Theorems,
we prove the main results in Section 3 and the sparse domination in Section 4.
%in Section 3  we prove the main results and in Section 4  the sparse domination.
Finally, in   Section 5 we give extra commentaries about  the $\mathcal{ A}_{A, p,q}$ classes of weights.

%the $M_{\alpha, A}$ and

\section{Preliminaries and results}

First we will start giving the  definitions of the fractional size and H\"{o}rmander
conditions for the kernels that we will be working along this paper. Let us introduce the following  notation. Let $1\leq r<\infty$, we set
$$\|f\|_{r,B}=%\sup_{B\ni x}
\left( \frac1{|B|}\int_B |f(x)|^r\;dx\right)^{1/r},$$
where $B$ is a
ball. Observe that in this averages the balls $B$ can be replaced
by cubes $Q$. The notation  $|x|\sim t$ means $t < |x|\leq 2t$ and
we write
$$\|f\|_{ r,|x|\sim t}=\|f \chi_{|x|\sim t}\|_{r,B(0,2t)}.$$

Let $0\leq \alpha < n$ and $1\leq r\leq \infty$.
The function $K_{\alpha}$ is said to satisfy the fractional size condition, % $S_{\alpha,\Psi}$, denote it by
$K_{\alpha}\in S_{\alpha,r}$, if there exists a constant $C>0$ such that
$$\|K_{\alpha}\|_{r,|x| \sim t} \leq C t^{\alpha - n}.$$

For $s=1$ we write $S_{\alpha,r}=S_{\alpha}$. Observe that if
$K_{\alpha}\in S_{\alpha}$, then there exists a constant $c>0$ such
that
$$\int_{|x|\sim t} |K_{\alpha}(x)|dx\leq c t^{\alpha}.$$

For $\alpha =0$ we write $S_{0, r}=S_{r}$.

% Let $\Psi$ a Young function and let $0\leq \alpha < n$.
The function  $K_{\alpha}$ satisfies the $L^{\alpha,r}$-H\"ormander
condition ($K_\alpha \in H_{\alpha,r}$), if there exist $c_{r}>1$
and $C_{r}>0$ such that for all $x$ and $R>c_{r}|x|$,
\begin{align*} \sum_{m=1}^{\infty} (2^mR)^{n- \alpha}  \| K_{\alpha}(\cdot - x) - K_{\alpha}(\cdot)\|_{r,|y|\sim2^mR} \leq C_{r}.\end{align*}
We say that $K_{\alpha} \in H_{\alpha,\infty}$ if $K_{\alpha}$ satisfies the previous condition with $\|\cdot\|_{L^{\infty},|x|\sim 2^mR}$ in place
of $\|\cdot\|_{r,|x|\sim 2^mR}$.
For $\alpha =0$ we write $H_{0, r}=H_{r}$, the classical
$L^r$-H\"ormander condition.

\begin{remark} Observe that if  $K_{\alpha}(x)=|x|^{n-\alpha}$ then
    $T_{\alpha}=I_{\alpha}$ the fractional integral and $K_\alpha\in S_{\alpha, \infty}\cap H_{\alpha, \infty}$.
    \end{remark}

\begin{remark}
Let $0\leq \alpha < n$, $m \in \mathbb{N}$ and $1\leq i \leq m$. Let $1<r_i<\infty$, $s\geq 1$ defined by  $\frac1{r_1}+\cdots+\frac1{r_m}+\frac1{s}=1$  and     $0\leq\alpha_i<n$ such that $\alpha_1+\cdots + \alpha_m=n-\alpha$.
If $k_i \in S_{n-\alpha_i, r_i}$ for $1\leq i\leq m$ then $K\in S_{\alpha}$ for $K(x,y)=k_1(x-A_1y)k_2(x-A_2y)\dots k_m(x-A_my)$. See details in \cite{IFRi18}.
  \end{remark}

%\begin{lema}\label{tamano}
 %   Let $0\leq \alpha < n$, $m \in \mathbb{N}$ and $1\leq i \leq m$. Let $1<r_i<\infty$, $s\geq 1$ defined by  $\frac1{r_1}+\cdots+\frac1{r_m}+\frac1{s}=1$  and
  %  $0\leq\alpha_i<n$ such that $\alpha_1+\cdots + \alpha_m=n-\alpha$.
  %  Suppose that the matrices $A_i$ satisfy  hypothesis $(H)$ and $k_i \in S_{n-\alpha_i, r_i}$ for $1\leq i\leq m$.
    %$k_i \in S_{n-\alpha_i, \Psi_i}$ for $1\leq i\leq m$.\\
    %Let $\varphi_k(t)=t\log(e+t)^k$ and let $\phi$ be  a Young function such that
    %$\Psi_1^{-1}(t)\cdots\Psi_m^{-1}(t)\overline{\varphi_k}^{-1}(t)\phi^{-1}(t)\lesssim
    %t$ for $t\geq t_0$, some $t_0>0$.

   % Let $K(x,y)=k_1(x-A_1y)k_2(x-A_2y)\dots k_m(x-A_my)$. Let $B=B(c_B,R)$ be a
  %  ball centred at $c_B$ with radius R.
  %  We write $\tilde{B}=B(c_B,2R)$ and for $1\leq i \leq m$, set $\tilde{B}_i=A_i^{-1}\tilde{B}$. If $z\in \tilde{B}_j$, for some $1\leq j \leq m$, then there exists a positive constant $C$ such that
   % $$ \int_{B} |K(y,z)|dy \leq CR^{\alpha}.$$
%\end{lema}

 To obtain a two weight  bound of the  integral operator $T_{\alpha,m}$ in \eqref{eq: defT}, we define the following testing constant for a  pair of weight $(u,v)$. Let $\D$ be a dyadic family, $1\leq r < \frac{n}{\alpha}$ and $1< p\leq q < \frac{n}{\alpha}$,
 $$\mathcal{T}_{A,r,out,\D}=\underset{R\in\D}{\sup} \;\; v(R)^{-1/p}\left\|\sum_{Q\in \D:R\subset Q}|Q|^{\alpha/n-1/r}\left(\int_Q v\chi_R\right)^{1/r}\chi_Q\right\|_{L^{q}(u_A)}<\infty,$$
 $${\mathcal{T}}^{*}_{A,r, out,\D}=\underset{R\in\D}{\sup} \;\; u_A(R)^{-1/q'}\left\|\sum_{Q\in \D:R\subset Q}|Q|^{\alpha/n-1/r}v(Q)^{\frac1{r}-1}\left(\int_Q u_A\chi_R\right)\chi_Q\right\|_{L^{p'}\left(v\right)}<\infty,$$
  $${\mathcal{T}}_{A,r,in,\D}=\underset{R\in\D}{\sup} \;\; v(R)^{-1/p}\left\|\sum_{Q\in \D:Q\subset R}|Q|^{\alpha/n-1/r}v(Q)^{1/r}\chi_Q\right\|_{L^{q}(u_A)}<\infty,$$
 $${\mathcal{T}}^{*}_{A,r,in,\D}=\underset{R\in\D}{\sup} \;\; u_A(R)^{-1/q'}\left\|\sum_{Q\in \D:Q\subset R}|Q|^{\alpha/n-1/r}v(Q)^{\frac1{r}-1}u_A(Q)\chi_Q\right\|_{L^{p'}\left(v\right)}<\infty.$$
 %where $\mathscr{S}$ is a $\eta$-sparse family of dyadic cubes, this family will be defined in the next section.

  We will also use the following notation:
   $${\mathcal{T}}_{r,out,\D}:={\mathcal{T}}_{I,r, out,\D},\,\, {\mathcal{T}}^{*}_{r,out,\D}:={\mathcal{T}}^{*}_{I,r, out,\D},\,\, {\mathcal{T}}_{r,in,\D}:={\mathcal{T}}_{I,r, in,\D} \text{ and  } {\mathcal{T}}^{*}_{r,in,\D}:={\mathcal{T}}^{*}_{I,r, in,\D}, $$
  where $I$ is the identity matrix.
  Also for $r=1$ $${\mathcal{T}}_{A,out,\D}:={\mathcal{T}}_{A,1, out,\D},\,\, {\mathcal{T}}^{*}_{A,out,\D}:={\mathcal{T}}^{*}_{A,1, out,\D},\,\, {\mathcal{T}}_{A,in,\D}:={\mathcal{T}}_{A,1, in,\D} \text{ and  } {\mathcal{T}}^{*}_{A,in,\D}:={\mathcal{T}}^{*}_{A,1, in,\D}.$$

  We consider   $3^n$ dyadic families $\{ \D_j\}$, defined in \cite{LN18},  with the following property: every bounded set is contained in some dyadic cube $Q\in \D_j$. In Theorem \ref{Sparse} we will give more details of these families.

 \begin{teo}\label{FuerteT}

    Let $0\leq \alpha < n$, $m \in \mathbb{N}$ and let  $T_{\alpha,m}$
    be the  integral operator defined by (\ref{eq: defT}). For $1\leq i
    \leq m$, let  $1<r_i\leq\infty$ and $0\leq\alpha_i<n$
    such that $\alpha_1+\cdots + \alpha_m=n-\alpha$. Let $k_i \in
    S_{n-\alpha_i, r_i}\cap H_{n-\alpha_i,r_i}$ and let the
    matrices $A_i$ satisfy  the hypothesis $(H)$.

    \noindent If $\alpha=0$,  suppose $T_{0,m}$ be of strong type $(p_0,p_0)$ for some $1<p_0<\infty$.

      Suppose that there exists $s\geq 1$ such that $\frac1{r_1}+\cdots +\frac1{r_m}+\frac1{s}=1$   and  $(u,v)$ are weights  such that  for the $3^n$ dyadic families, $\{\D_j\}$, the testing constants  ${\mathcal{T}}_{A_i,s,out,\D_j},{\mathcal{T}}^{*}_{A_i,s, out,\D_j}<\infty$ for all $1\leq i \leq m$, $1\leq s < p <\frac{n}{\alpha}$. Then, there exists $c>0$ not depending on $f$, and  the pair $(u,v)$ such that
    \[ \|T_{\alpha,m}(f\sigma)\|_{L^q(u)} \leq c \|f\|_{L^p(\sigma)} \sum_{j=1}^{3^n}\sum_{i=1}^m ({\mathcal{T}}_{A_i,s,out,\D_j} +{\mathcal{T}}^{*}_{A_i,s, out,\D_j}),
    \]

 where $\sigma= v^{\frac{p'}{(p/s)'}}$.
 \end{teo}

It can be proved an analogous result with the conditions ${\mathcal{T}}_{A_i,s,in,\D_j},{\mathcal{T}}^{*}_{A_i,s, in,\D_j}$ in place of ${\mathcal{T}}_{A_i,s,out,\D_j},{\mathcal{T}}^{*}_{A_i,s, out,\D_j}$.

In  the case $r_i=\infty$, for all $1\leq i \leq m$, we obtain for the  integral operator $T_{\alpha, m}$, a new two weighted result:
 \begin{teo}\label{FuerteT1}
    Let $0\leq \alpha < n$, $m \in \mathbb{N}$ and let  $T_{\alpha,m}$
    be the integral operator defined by (\ref{eq: defT}). For $1\leq i
    \leq m$, let  $0\leq\alpha_i<n$
    such that $\alpha_1+\cdots + \alpha_m=n-\alpha$. Let $k_i \in
    S_{n-\alpha_i, \infty}\cap H_{n-\alpha_i,\infty}$ and let the
    matrices $A_i$ satisfy  the hypothesis $(H)$.

    \noindent If $\alpha=0$,  suppose $T_{0,m}$ be of strong type $(p_0,p_0)$ for some $1<p_0<\infty$.\\
Let $1<p\leq q < \frac{n}{\alpha}$.     If    $(u,v)$  are pairs of weights such that $(u_{A_i},v)\in \mathcal{M}_{\alpha, p,q}$ and $(v,u_{A_i})\in \mathcal{M}_{\alpha,q',p'}$    for $i=1,\dots,m$.
 %  \begin{equation} \label{Macot1}
% $$    M_{\alpha}(\cdot v ): L^{p}(v)\to L^q(u_{A_i}),$$
 %  \end{equation}
 %  and
 %  \begin{equation} %\label{Macot2}
% $$    M_{\alpha}(\cdot u_{A_i}):  L^{q'}(u_{A_i})\rightarrow L^{p'}(v),$$
 %  \end{equation}
    Then,
    $$T_{\alpha,m}:L^p(v)\to L^q(u).$$
 \end{teo}

 Given $w$ a weight such that  the pair $(u,v)=(w^q_{A_i}, w^{-p'})$ we can write the following
 \begin{corol}\label{CorolFuerteT1}
    Let $0\leq \alpha < n$, $m \in \mathbb{N}$ and let  $T_{\alpha,m}$
    be the  integral operator defined by (\ref{eq: defT}). For $1\leq i
    \leq m$, let  $0\leq\alpha_i<n$
    such that $\alpha_1+\cdots + \alpha_m=n-\alpha$. Let $k_i \in
    S_{n-\alpha_i, \infty}\cap H_{n-\alpha_i,\infty}$ and let the
    matrices $A_i$ satisfy  the hypothesis $(H)$.

    \noindent If $\alpha=0$,  suppose $T_{0,m}$ be of strong type $(p_0,p_0)$ for some $1<p_0<\infty$.\\
    Let $1<p< \frac{n}{\alpha}$ and $\frac1{q}=\frac1{p}-\frac{\alpha}{n}$. If    $w$  is a weight such that $w \in \mathcal{M}_{\alpha,A_i, p,q} \cap \mathcal{M}_{\alpha,A_i, q',p'}$, for $i=1,\dots,m$.
    %\begin{equation} \label{Macot1}
    %M_{\alpha,A^{-1}_{i}}(\cdot ): L^{p}(w^{p})\to L^q(w^q),
    %\end{equation}
    %and
    %\begin{equation} \label{Macot2}
    %M_{\alpha}(\cdot w^q_{A_i}):  L^{q'}(w^q_{A_i})\rightarrow L^{p'}(w^{-p'}),
    %\end{equation}
    Then,
    $$T_{\alpha,m}:L^p(w^{p})\to L^q(w^q).$$

 \end{corol}

  \begin{remark}
    By Remark \ref{RemtestimplicaAApq},  if  $\displaystyle w\in\bigcap_{i=1}^m\mathcal{M}_{\alpha,A_i, p,q}$ then $\displaystyle w \in \bigcap_{i=1}^m\mathcal{A}_{A_i,p,q}$  and this implies  $w(A_ix)\leq c w(x)$ a.e. $x\in \mathbb{R}^n$, for all $1\leq i\leq m$. (See  Proposition \ref{PropwA}). In \cite{RU14} it was proved this same result under the hypothesis $w\in \mathcal{A}_{p,q}$ and $w(A_ix)\leq c w(x)$ a.e. $x\in \mathbb{R}^n$ for all $1\leq i\leq m$. As we say in Corollary \ref{CaracPesosMA}, $w\in \mathcal{A}_{p,q}$ and $w(Ax)\leq c w(x)$ then $w\in \mathcal{M}_{\alpha,A, p,q}$  and this implies $w\in \mathcal{A}_{A,p,q}$. Therefore we obtain a different proof without using essentially $w\in \mathcal{A}_\infty$. In other words, we have
$$A_{p,q}\cap \{w: w_A \lesssim w\}\subset \{w : {\text{\it testing condition for }} M_{\alpha,A^{-1}} \} \subset \mathcal{A}_{A,p,q}.$$
 \end{remark}

 \section{Proof of the main results}

In the following proof we give a characterization of the good weights stated in Theorem  \ref{Ejem}.
 \begin{proof}[Proof of Theorem \ref{Ejem}]
     Let $B=B(c_B,R)$ and $B_i=A_i^{-1}B$. %$B_i=B(A_i^{-1}c_B,R)$. %.
  Suppose that $f=\chi_{B_1}$ %or $f=\chi_{B_2}$ 
and
    \[ T_{\alpha,2}(\cdot v) : L^p(v) \to L^q(u).
    \]
    Then
    \[ v(\chi_{B_1})^{-1/p}\left(\int T_{\alpha,2}(\chi_{B_1} v)(x)^qu(x)dx\right)^{1/q}<\infty.
    \]
If $x\in B$ and $y\in B_1$, then
    \[ |x-A_1y|\leq |x-c_B|+|c_B-A_1y|\leq R+C_A R=(1+C_A)R.
    \]

    If $|x-A_2y|\leq |x-A_1y|$, then
    \[ |x-A_1y|,|x-A_2y|\lesssim R,
    \]
    and  since $\alpha_1+\alpha_2=n-\alpha$ then
    \[ \frac{v(y)}{|x-A_1y|^{\alpha_1}|x-A_2y|^{\alpha_2}}\geq \frac{v(y)}{|x-A_1y|^{n-\alpha}}\geq C
    \frac{v(y)}{R^{n-\alpha}}.
    \]

    If $|x-A_1y|\leq |x-A_2y|$, then
    \[ \frac{v(y)}{|x-A_1y|^{\alpha_1}|x-A_2y|^{\alpha_2}}\geq
    \frac{v(y)}{|x-A_2y|^{n-\alpha}}.
    \]

    If $2^j|x-A_1y|\leq |x-A_2y|\leq 2^{j+1}|x-A_1y|$, then
    \[
    \frac1{|x-A_2y|^{n-\alpha}}\geq 2^{(\alpha-n)(j+1)}\frac1{|x-A_1y|^{n-\alpha}}\geq 2^{(\alpha-n)(j+1)}\frac1{R^{n-\alpha}}.
    \]

   Using that \[ \sum_{j=1}^{\infty}
    \left(2^{\alpha-n}\right)^{j}=\frac1{1-2^{\alpha-n}}-1=\frac{2^{\alpha-n}}{1-2^{\alpha-n}},
    \]
    in the case that  $|x-A_1y|\leq |x-A_2y|$, we have
    \[ \frac{v(y)}{|x-A_1y|^{\alpha_1}|x-A_2y|^{\alpha_2}}\geq  2^{\alpha-n}\frac{2^{\alpha-n}}{1-2^{\alpha-n}}
    \frac{v(y)}{R^{n-\alpha}}.
    \]
    Hence, if $x\in B$ and $y\in B_1$,
 \[ \frac{v(y)}{|x-A_1y|^{\alpha_1}|x-A_2y|^{\alpha_2}}\geq  C_{n,\alpha,A} \frac{v(y)}{R^{n-\alpha}}.
    \]
    We have an analogous result if $y\in B_2$.

    If $x\in B$
    \begin{align*}
    T_{\alpha,2}(\chi_{B_1} v)(x)\geq R^{\alpha-n}v(B_1)=|B|^{\alpha/n-1}v(B_1).
    \end{align*}
    Then we have
    \begin{align*}
    v(B_1)^{-1/p}&\left(\int T_{\alpha,2}(\chi_{B_1} v)(x)^qu(x)dx\right)^{1/q}
    \geq v(B_1)^{-1/p}\left(\int_B T_{\alpha,2}(\chi_{B_1} v)(x)^qu(x)dx\right)^{1/q}
    \\&\geq v(B_1)^{-1/p}\left(\int_B |B|^{q(\alpha/n-1)}v(B_1)^qu(x)dx\right)^{1/q}
    \\&\geq v(B_1)^{-1/p} |B|^{\alpha/n-1}v(B_1)u(B)^{1/q}
    \\&=|det A_1^{-1}|^{-1/p'}|B|^{1/q} u(B)^{1/q}|B_1|^{1/p'}v(B_1)^{1/p'}
    \\&=|det A_1|^{1/p'}|B|^{\alpha/n-1} u(B)^{1/q}v_{A_1^{-1}}(B)^{1/p'}.
    \end{align*}

If we take $f=\chi_{B_2}$, in an analogous way we have
$$|B|^{\alpha/n-1} u(B)^{1/q}v_{A_2^{-1}}(B)^{1/p'}<\infty.$$

    If $u=w^q$ and $v=w^{-p'}$ we conclude that if $T_{\alpha,2} : L^p(w^p) \to L^q(w^q)$ then
    $w\in \cap_{i=1}^2\mathcal{A}_{p,q,{A_i^{-1}}}=\cap_{i=1}^2\mathcal{A}_{A_i,p,q}$.

Furthermore, consider the  case  $A_1=A$ and $A_2=A^{-1}$.
    If $w\in \mathcal{A}_{A,p,q} \cap \mathcal{A}_{A^{-1},p,q}$ then  $w_A\sim w$ and  $w\in \mathcal{A}_{p,q}$. Therefore $T_{\alpha,2}:L^p(w^p)\to L^q(w^q)$.
Now for  the case  $A_1=-I$ and $A_2=I$, if $w\in \mathcal{A}_{-I,p,q} \cap \mathcal{A}_{p,q}$ then $w_{-I}\simeq w$ and $T_{\alpha, 2}:L^p(w^p)\to L^q(w^q)$ (see \cite{RU14}).
 \end{proof}

 To prove the boundedness of $T_{\alpha,m}$ for general kernels  we will use an appropriate sparse domination.

 Given a cube $Q\in \R$, we denote by $\mathcal{D}(Q)$ the family of all dyadic cubes  respect to $Q$, that is, the cube obtained subdividing  repeatedly $Q$ and each of its descendant into $2^n$ subcubes of the same side lengths.

 Given a dyadic family $\mathcal{D}$ we say that a subfamily $\mathscr{S}\subset \mathcal{D}$ is a $\eta$-sparse family with $0<\eta<1$, if  for every $Q\in \mathscr{S}$, there exists a measurable set $E_Q\subset Q$ such that $\eta|Q|\leq |E_Q|$ and the family $\{E_Q\}_{Q\in \mathscr{S}}$   are pairwise disjoint.\\
  The following theorem will be prove in Section \ref{secsparse}.
%   {\red{la demostraci\'on quedo fuera del documento, volver a citar lo de los compactos}}
 \begin{teo}\label{Sparse}
    Let $0\leq \alpha < n$, $m \in \mathbb{N}$ and let  $T_{\alpha,m}$
    be the  integral  operator defined by (\ref{eq: defT}). For $1\leq i
    \leq m$, let $1<r_i\leq\infty$,  defined by $s\geq 1$, $\frac1{r_1}+\cdots+\frac1{r_m}+\frac1{s}=1$ and  $0\leq\alpha_i<n$ such that $\alpha_1+\cdots + \alpha_m=n-\alpha$. Let $k_i \in
    S_{n-\alpha_i, r_i}\cap H_{n-\alpha_i, r_i}$ and let the
    matrices $A_i$ satisfy  the hypothesis $(H)$.

    \noindent If $\alpha=0$,  suppose $T_{0,m}$ be of strong type $(p_0,p_0)$ for some $1<p_0<\infty$.\\
    There exist $c>0$ %=C(n,\alpha,A_1,...,A_m)
    and    $3^n$    $\frac1{2.9^n}$-sparse families, $\{{\mathscr{S}_j}\}_{j=1}^{3^n}$,
    such that   for $f\in
    L_c^{\infty}(\R)$ and a.e. $x\in \R$, we have
  %  \begin{equation}
   % |T_{\alpha,m}f(x)|\leq c \sum_{j=1}^{3^n}\sum_{i=1}^m \sum_{Q\in \mathscr{S}_j}
  %  |Q|^{\alpha/n}\|f\|_{s,Q}\chi_Q(A_i^{-1}x).
  %  \end{equation}
 %If we define the sparse operator $\mathcal{A}_{\alpha,s,\mathscr{S}}$ by
 % $$\mathcal{A}_{\alpha,s,\mathscr{S}}f(A_i^{-1}x)=\sum_{Q\in \mathscr{S}}
 %|Q|^{\alpha/n}\|f\|_{s,Q}\chi_Q(A_i^{-1}x),$$
%then
   % \begin{equation}
   % |T_{\alpha,m}(f)(x)| \leq c \sum_{j=1}^{3^n}\sum_{i=1}^m \mathcal{A}_{\alpha, s,\mathscr{S}_j}f(A_i^{-1}x).
  %  \end{equation}
     $$
    |T_{\alpha,m}(f)(x)| \leq c \sum_{j=1}^{3^n}\sum_{i=1}^m \mathcal{A}_{\alpha, s,\mathscr{S}_j}f(A_i^{-1}x),
    $$
where  $\mathcal{A}_{\alpha,s,\mathscr{S}}$ is the sparse operator defined by 
$$\mathcal{A}_{\alpha,s,\mathscr{S}}f(A_i^{-1}x):=\sum_{Q\in \mathscr{S}}
|Q|^{\alpha/n}\|f\|_{s,Q}\chi_Q(A_i^{-1}x).$$
     \end{teo}

 %Now, we consider the case $\varphi(t)=t^r$ some $r\geq 1$. In this case we prove the following

 To  prove the boundedness of $T_{\alpha,m}$ it is enough to show the boundedness of $\mathcal{A}_{\alpha,s,\mathscr{S}_j}$.
 %\begin{teo}\label{SparseTesting}
 %  Let $0\leq \alpha < n$, $1<r<p{\red\leq q<}\infty$  %$\frac1{q}=\frac1{p}-\frac{\alpha}{n}$.
 %  Let $A$ be a invertible matrix and $\mathscr{S}$ be a sparse family. Let $(u, \sigma)$ a pair of weights and $v=\sigma^{r\frac{(p/r)'}{p'}}$. If $(u,\sigma)$ satisfies local testing conditions, ${\mathcal{T}}_{r,in}, {\mathcal{T}}^*_{r,in} <\infty$,
    %$$\tilde{\mathcal{T}}_{in}=\underset{R\in\mathscr{S}}{\sup} \;\; v(R)^{-1/p}\left\|\sum_{Q\in \mathscr{S}:Q\subset R}|Q|^{\alpha/n-1/r}v(Q)^{1/r}\chi_Q\right\|_{L^{q}(u)}<\infty,$$
    %$$\tilde{\mathcal{T}}^{*}_{in}=\underset{R\in\mathscr{S}}{\sup} \;\; u(R)^{-1/q'}\left\|\sum_{Q\in \mathscr{S}:Q\subset R}|Q|^{\alpha/n-1/r}v(Q)^{\frac1{r}-1}u(Q)\chi_Q\right\|_{L^{p'}\left(v\right)}<\infty.$$
 %  then, $$\left(\int_{\mathbb{R}^{n}}A_{\alpha,r,\mathscr{S}}(f\sigma)^qu\right)^{1/q}\leq C_{n,p,\alpha}({\mathcal{T}}_{r,in}+{\mathcal{T}}^*_{r,in})\left(\int_{\mathbb{R}^{n}}|f|^p\sigma\right)^{1/p}.$$
 %  If ${\mathcal{T}}^*_{r,in}<\infty$ then
 %  $$\mathcal{A}_{\alpha,r,\mathscr{S}}:L^p(\sigma)\rightarrow L^{q,\infty}(u).$$

 %  Furthermore, if $u=w_A^q$ and $\sigma=w^{-p'}$, then $v=w^{-r(p/r)'}$ and
 %  $$\mathcal{A}_{\alpha,r,\mathscr{S}}:L^p(w^p)\rightarrow L^{q}(w_A^q).$$
 %\end{teo}

  \begin{teo}\label{SparseTesting}
    Let $0\leq \alpha < n$, $1<r<p\leq q<\infty$  %$\frac1{q}=\frac1{p}-\frac{\alpha}{n}$.
    Let $A$ be a invertible matrix and $\mathscr{S}$ be a sparse family from a family of dyadic cubes $\D$. Let $(u, v)$ a pair of weights and $\sigma=v^{\frac{p'}{(p/r)'r}}$. If $(u,v)$ satisfies local testing conditions, ${\mathcal{T}}_{r,out,\D}, {\mathcal{T}}^*_{r,out,\D} <\infty$,
    %$$\tilde{\mathcal{T}}_{in}=\underset{R\in\mathscr{S}}{\sup} \;\; v(R)^{-1/p}\left\|\sum_{Q\in \mathscr{S}:Q\subset R}|Q|^{\alpha/n-1/r}v(Q)^{1/r}\chi_Q\right\|_{L^{q}(u)}<\infty,$$
    %$$\tilde{\mathcal{T}}^{*}_{in}=\underset{R\in\mathscr{S}}{\sup} \;\; u(R)^{-1/q'}\left\|\sum_{Q\in \mathscr{S}:Q\subset R}|Q|^{\alpha/n-1/r}v(Q)^{\frac1{r}-1}u(Q)\chi_Q\right\|_{L^{p'}\left(v\right)}<\infty.$$
    then, $$\left(\int_{\mathbb{R}^{n}}\mathcal{A}_{\alpha,r,\mathscr{S}}(f\sigma)^qu\right)^{1/q}\leq C_{n,p,\alpha}({\mathcal{T}}_{r,out,\D}+{\mathcal{T}}^*_{r,out,\D})\left(\int_{\mathbb{R}^{n}}|f|^p\sigma\right)^{1/p}.$$
    If ${\mathcal{T}}^*_{r,out,\D}<\infty$ then
    $$\mathcal{A}_{\alpha,r,\mathscr{S}}:L^p(\sigma)\rightarrow L^{q,\infty}(u).$$

    Furthermore, if $u=w_A^q$ and $v=w^{-r(p/r)'}$, then $\sigma=w^{-p'}$ and
    $$\mathcal{A}_{\alpha,r,\mathscr{S}}:L^p(w^p)\rightarrow L^{q}(w_A^q).$$
 \end{teo}

 There is an analogous result with global testing conditions, ${\mathcal{T}}_{r, in,\D}$ and ${\mathcal{T}}^{*}_{r,in,\D}$.

 %$$\tilde{\mathcal{T}}_{out}=\underset{R\in\mathscr{S}}{\sup} \;\; v(R)^{-1/p}\left\|\sum_{Q\in \mathscr{S}:R\subset Q}|Q|^{\alpha/n-1/r}v(Q)^{1/r}\chi_Q\right\|_{L^{q}(u)}<\infty,$$
 %$$\tilde{\mathcal{T}}^{*}_{out}=\underset{R\in\mathscr{S}}{\sup} \;\; u(R)^{-1/q'}\left\|\sum_{Q\in \mathscr{S}:R\subset Q}|Q|^{\alpha/n-1/r}v(Q)^{\frac1{r}-1}u(Q)\chi_Q\right\|_{L^{p'}\left(v\right)}<\infty.$$

 %\begin{remark}
 %  Observe that if we consider $u=u_A$ in  $\tilde{\mathcal{T}}_{out}$ and $\tilde{\mathcal{T}}^{*}_{out}$ we obtain the conditions $\tilde{\mathcal{T}}_{A,r,out}$ and $\tilde{\mathcal{T}}^{*}_{A,r,out}$, respectively.
% \end{remark}

 We consider the following sparse operator defined in \cite{FH18}, for $\mathscr{S}$ a sparse family, $0<t<\infty$ and $0<\beta\leq 1$,

 $$\tilde{\mathcal{A}}_{t,\mathscr{S}}^{\beta}g(x)=\left(\sum_{Q\in\mathscr{S}} \left(|Q|^{-\beta}\int_Q g\right)^t\chi_{Q}(x)\right)^{1/t}.$$

 Observe that
 \begin{equation}\label{CompSparse}
 \mathcal{A}_{\alpha,r,\mathscr{S}}(f)=\left(\tilde{\mathcal{A}}_{1/r,\mathscr{S}}^{1-r\alpha/n}(f^r)\right)^{1/r}.
 \end{equation}
 The next lemma follows the same proof as Proposition 3.1  in \cite{FH18} considering the testing constant $\tilde{\mathcal{T}}_{t, out}$ in place of the testing constant $\tilde{\mathcal{T}}_{t, in}$.

%{\red{poner la dependencia de $\D$ en la constante $\tilde{\mathcal{T}}_{t, out}$???}}
 \begin{lema}\cite{FH18}\label{SparseTilde}
    Let $1<p\leq q <\infty$, $t\in(0,p)$, $\beta\in (0,1]$, $\mathscr{S}$ be a sparse collection of dyadic cubes and let $(u,v)$ be a pair of weights. Define the testing constants
    \begin{align*}
    \tilde{\mathcal{T}}_{t, out}&:=\sup_{R\in \mathscr{S}} v(R)^{-t/p}\left\| \sum_{Q\in \mathscr{S} : R\subset Q} |Q|^{-\beta t} \left(\int_Q v\chi_R\right)^t \chi_Q\right\|_{L^{q/t}(u)},\\
    \tilde{\mathcal{T}}^*_{t,out}&:=\sup_{R\in \mathscr{S}} u(R)^{-1/(q/t)'}\left\| \sum_{Q\in \mathscr{S} : R\subset Q}|Q|^{-\beta t} v(Q)^{t-1} \left(\int_Q u\chi_R\right) \chi_Q\right\|_{L^{(p/t)'}(v)}.
    \end{align*}
    Then,
    \[
    \|\tilde{\mathcal{A}}_{t,\mathscr{S}}^{\beta}(v \cdot)\|^t_{L^p(v)\to L^q(u)}\lesssim \tilde{\mathcal{T}}_{t,out}+\tilde{\mathcal{T}}^*_{t,out}.
    \]
 \end{lema}

%\begin{remark} If $\mathscr{S}$ be a sparse family from a family of dyadic cubes $\D$. Then, it is easy to see that $\tilde{\mathcal{T}}_{t, out}\leq \tilde{\mathcal{T}}_{t, out, \D}$ and $\tilde{\mathcal{T}}^*_{t, out}\leq \tilde{\mathcal{T}}^*_{t, out, \D}$, {\red defined in the previous section}.
%\end{remark}

 \begin{proof}[Proof of Theorem \ref{SparseTesting}]
    %{\vb Se usa ideas de \cite{FH18} \cite{UTLS12}}

    Using \eqref{CompSparse}, we have
    \begin{align*}
    \left(\int_{\mathbb{R}^{n}}\mathcal{A}_{\alpha,r,\mathscr{S}}(f\sigma)^qu\right)^{1/q}
    &=\left(\int_{\mathbb{R}^{n}}\left(\tilde{\mathcal{A}}_{\mathscr{S}}^{1/r,1-\alpha/n}(f^r\sigma^r)\right)^{q/r}u\right)^{1/q}
    \\&=\left(\int_{\mathbb{R}^{n}}\left(\tilde{\mathcal{A}}_{\mathscr{S}}^{1/r,1-\alpha/n}(f^r\sigma^rv^{-1}v)\right)^{q/r}u\right)^{1/q}.
    \end{align*}
    If $(u,v)$ satisfies the testing constants  $\tilde{\mathcal{T}}_{1/r,out},\tilde{\mathcal{T}}^*_{1/r,out}$ with $p/r,q/r$ and $\beta=1-\alpha/n$, therefore, by Lemma \ref{SparseTilde}, we have
        \begin{align*}
    \left(\int_{\mathbb{R}^{n}}\mathcal{\mathcal{A}}_{\alpha,r,\mathscr{S}}(f\sigma)^qu\right)^{1/q}
    &=\left(\int_{\mathbb{R}^{n}}\left(\tilde{\mathcal{A}}_{\mathscr{S}}^{1/r,1-\alpha/n}(f^r\sigma^rv^{-1}v)\right)^{q/r}u\right)^{(r/q) 1/r}
    \\&\leq C_{n,p,\alpha}(\tilde{\mathcal{T}}_{1/r,out}+\tilde{\mathcal{T}}^*_{1/r,out})\left(\int_{\mathbb{R}^{n}}|f^r\sigma^rv^{-1}|^{p/r}v\right)^{(r/p) 1/r}
    \\&= C_{n,p,\alpha}(\tilde{\mathcal{T}}_{1/r,out}+\tilde{\mathcal{T}}^*_{1/r,out})\left(\int_{\mathbb{R}^{n}}|f|^p\sigma^pv^{1-p/r}\right)^{1/p}.
    \end{align*}
    As $v=\sigma^{r\frac{(p/r)'}{p'}}$, then
    \begin{align*}
    \left(\int_{\mathbb{R}^{n}}\mathcal{A}_{\alpha,r,\mathscr{S}}(f\sigma)^qu\right)^{1/q}
    = C_{n,p,\alpha}(\tilde{\mathcal{T}}_{1/r,out}+\tilde{\mathcal{T}}^*_{1/r,out})\left(\int_{\mathbb{R}^{n}}|f|^p\sigma\right)^{1/p}.
    \end{align*}
    Observe that if the pair $(u,v)=(u,\sigma^{r\frac{(p/r)'}{p'}})$ satisfies the testing constants ${\mathcal{T}}_{r,out,\D},{\mathcal{T}}^*_{r,out,\D}$,  with $\alpha$, $p,\,q$ then the pair  $(u,v)$ satisfies the testing constants $\tilde{\mathcal{T}}
    _{1/r,out},\tilde{\mathcal{T}}^*_{1/r,out}$ with  $\beta=1-\alpha/n$,  $p/r$ and $q/r$.  Moreover $\tilde{\mathcal{T}}
    _{1/r,out}{ \leq }{\mathcal{T}}_{r,out,\D}$ and $\tilde{\mathcal{T}}^*
    _{1/r,out}{ \leq }{\mathcal{T}}^*_{r,out,\D}$.

    %\begin{align*}
    %\tilde{\mathcal{T}}_{r,in}=\underset{R\in\mathcal{S}}{\sup} \;\; \sigma^{r\frac{(p/r)'}{p'}}(R)^{-1/p}\|\sum_{Q\in \mathcal{S}:Q\subset R}|Q|^{-(1-\alpha/n)/r}\sigma^{r\frac{(p/r)'}{p'}}(Q)^{1/r}\chi_Q\|_{L^{q}(u)}
    %\end{align*}

    %$$\tilde{\mathcal{T}}^{*}_{r,in}=\underset{R\in\mathcal{S}}{\sup} \;\; u(R)^{-1/p'}\|\sum_{Q\in \mathcal{S}:Q\subset R}|Q|^{-(1-\alpha/n)/r}\sigma^{r\frac{(p/r)'}{p'}}(Q)^{\frac1{r}-1}u(Q)\chi_Q\|_{L^{p'}\left(\sigma^{r\frac{(p/r)'}{p'}}\right)}.$$
    %Then
    Therefore, we get
    \begin{align*}
    \left(\int_{\mathbb{R}^{n}}\mathcal{A}_{\alpha,r,\mathscr{S}}(f\sigma)^qu\right)^{1/q}
    \leq C_{n,p,\alpha}({\mathcal{T}}_{r,out,\D}+{\mathcal{T}}^*_{r,out,\D})\left(\int_{\mathbb{R}^{n}}|f|^p\sigma\right)^{1/p}.
    \end{align*}

 If we consider the testing constants  ${\mathcal{T}}_{r,in,\D}$ and ${\mathcal{T}}^{*}_{r,in,\D}$, the proof is similar, using ideas in \cite{LSUT10} %{\red to obtain an analogous result as the previous lemma. saca\'{\i}ia esta frase}
 \end{proof}

 \begin{proof}[Proof of Theorem \ref{FuerteT}]
    By hypothesis and using Theorem \ref{Sparse} we have that,
    \[
    |T_{\alpha,m}f(x)|\leq c \sum_{j=1}^{3^n}\sum_{i=1}^m \mathcal{A}_{\alpha,s,\mathscr{S}_j}f(A_i^{-1}x),%\sum_{Q\in \mathscr{S}_i} |Q|^{\alpha/n}\|f\|_{s,Q}\chi_Q(A_i^{-1}x)
    \]
    then
    \[ \|T_{\alpha,m}(f\sigma)\|_{L^q(u)} \leq c \sum_{j=1}^{3^n}\sum_{i=1}^m \|\mathcal{A}_{\alpha,s,\mathscr{S}_j}(f\sigma)\|_{L^q(u_{A_i})}.
    \]
    Since ${\mathcal{T}}_{A_i,s,out,\D_j},{\mathcal{T}}^{*}_{A_i,s, out,\D_j}<\infty$, for $1\leq i\leq m$ and $1\leq j \leq 3^n $, by Theorem \ref{SparseTesting} we get,
    \[\|\mathcal{A}_{\alpha,s,\mathscr{S}_j}(f\sigma)\|_{L^q(u_{A_i})}\lesssim ({\mathcal{T}}_{A_i,s,out,\D_j} +{\mathcal{T}}^{*}_{A_i,s, out,\D_j}) \|f\|_{L^p(\sigma)},
    \]
    then
    \[ \|T_{\alpha,m}(f\sigma)\|_{L^q(u)} \leq c \|f\|_{L^p(\sigma)} \sum_{j=1}^{3^n}\sum_{i=1}^m ({\mathcal{T}}_{A_i,s,out,\D_j} +{\mathcal{T}}^{*}_{A_i,s, out,\D_j}).
    \]
 \end{proof}

 For the proof of Theorem \ref{FuerteT1} we need the following results:

 \begin{lema} Let $(u,v)$  be a  pair of weight such that $(u_{A},v)\in \mathcal{M}_{\alpha, p,q}$ and $(v,u_{A})\in \mathcal{M}_{\alpha,q',p'}$, then, ${\mathcal{T}}_{A,out,\D}<\infty$ and  ${\mathcal{T}}^{*}_{A,out,\D}<\infty$ respectively, for any family of dyadic cubes $\D$.
 \end{lema}

 \begin{proof}
    We will only see that ${\mathcal{T}}_{A,out,\D}<\infty$, the other case it is prove in a similar way.

% Fix a cube $Q$ and $x\in \R$ such that there exists a dyadic cube  {\red $P\in \D$ with $x\in P$ and $Q\subsetneq P$ }. Let $Q_0$

 Let $R$ be a cube and $x\in R$, let $Q_k \in \D$ such that $R\subset Q_k$ then $l(Q_k)=2^k l(R)$,
\begin{align*}
 \sum_{Q\in \D:R\subset Q}|Q|^{\alpha/n-1/r}\left(\int_{Q} v\chi_R\right)^{1/r}\chi_Q(x)
 &= \sum_{k=0}^{\infty} |Q_k|^{\alpha/n-1/r}\left(\int_{Q_k} v\chi_R\right)^{1/r}\chi_{Q_k}(x)
  \\&=|R|^{\alpha/n-1/r}\left(\int_{R} v\right)^{1/r} \sum_{k=0}^{\infty} 2^{k(\alpha/n-1/r)}
  \\&\leq C M_{\alpha,r}(v\chi_R)(x)\chi_R(x).
 \end{align*}
Then,
$$\mathcal{T}_{A,r,out,\D} \leq C \underset{R\in\D}{\sup} \;\; v(R)^{-1/p}\left\|M_{\alpha,r}(v\chi_R)(x)\chi_R(x)\right\|_{L^{q}(u_A)}=C[u_A,v]_{\mathcal{M}_{\alpha, p,q}}. $$

 \end{proof}

 \begin{proof}[Proof of Theorem \ref{FuerteT1}]
 The proof follows from Theorem \ref{FuerteT} and the previous Lemma.

 \end{proof}

 %If $\phi(t)=t$, the Orlicz spaces asociated to $\{\Psi_i\}_{i}$ are $L^{\infty}$, (i.e. $\Psi_i\equiv \infty$). We prove the following

 \begin{proof}[Proof of Theorem \ref{APart}]
 %
 %Then, by \cite{IFRi18} we have
 %  $$T_{\alpha,m}:L^p(w^p)\to L^q(w^q).$$

    Using the sparse domination, Theorem \ref{Sparse}, we have
    \begin{align}\label{eq: desig1}
    \|T_{\alpha, m}f\|_{L^q(w^q)}\leq C \sum_{j=1}^{3^n}\sum_{i=1}^m \| \mathcal{A}_{\alpha, r,\mathscr{S}_j}f\|_{L^q(w_{A_i}^q)}.
    \end{align}
    Since $A_j=A_i^{-1}$ and $w^s\in \bigcap_{i=1}^m\mathcal{A}_{A_i,\frac{p}{s},\frac{q}{s}}$, then $w\in A_{\frac{p}{s},\frac{q}{s}}$ and $w_{A_i}\simeq w$, $w_{A_l}\lesssim w$ for $l\not=i,j$. So, we have $w^q,w^{-s(p/s)'}\in A_{\infty}$.
    In the other hand, let $A$ be a invertible matrix, if $w^s\in\mathcal{A}_{A,\frac{p}{s},\frac{q}{s}}$ then the pair $(w_A^s,w^{-s})$ satisfies the $A_{\frac{p}{s},\frac{q}{s}}$ condition.

    Since $(w_A^s,w^{-s})$ satisfies the $A_{\frac{p}{s},\frac{q}{s}}$ condition  and $w^q,w^{-s(p/s)'}\in A_{\infty}$, we obtain
        \begin{align}\label{eq: desig2}
        \|\mathcal{A}_{\alpha,s,\mathscr{S}}f\|_{L^q(w_A^q)}\leq c_n [w^s]_{A_{A,\frac{p}{s},\frac{q}{s}}}^{\max\left\{1-\frac{\alpha}{n},\frac{(p/s)'}{q}\left(1-\frac{\alpha s}{n}\right)\right\}} \|f\|_{L^p(w^p)},
        \end{align}
 and the exponent is sharp. The proof of this inequality is analogous to the proved in \cite{IFRV18}.

 Using \eqref{eq: desig1} and \eqref{eq: desig2}, we have
 \begin{align*}
    \|T_{\alpha, m}f\|_{L^q(w^q)}\leq C\|f\|_{L^p(w^p)} \sum_{i=1}^m  [w^s]_{A_{A_i,\frac{p}{s},\frac{q}{s}}}^{\max\left\{1-\frac{\alpha}{n},\frac{(p/s)'}{q}\left(1-\frac{\alpha s}{n}\right)\right\}}.
 \end{align*}

 \end{proof}

\section{Proof of sparse domination}\label{secsparse}
In this section, in the proof of Lemmas \ref{lemaalpha} and \ref{lemasparse2},  we consider only $m=2$. We will write for $\alpha\geq 0$, $T_\alpha:=T_{\alpha,2}$ and   also $T=T_{0}$. For general case, the results and proofs are analogous.

First, we need some end-point estimates for the maximal operator
$M_{T_{\alpha}}$,
%Let $0 \leq \alpha < n$ and $T_{\alpha}$ be the operator defined by \eqref{eq: defT}, then we define
the grand maximal truncated of $T_{\alpha}$, is defined by
\begin{align*}
M_{T_{\alpha}}f(x)=\underset{\underset{Q_2\ni A_2^{-1}x}{Q_1\ni
A_1^{-1}x}}{\sup} \; \underset{\xi \in Q_1\cup Q_2
}{\sup\text{ess}}|T_\alpha(f\chi_{\R\setminus 3(Q_1\cup Q_2)})(\xi)|,
\end{align*}
and the local version
\begin{align*}
M_{T_{\alpha},Q_0^{1}\cup
Q_0^{2}}f(x)=\underset{\underset{Q_0^{2}\subset Q_2\ni
A_2^{-1}x}{Q_0^{1}\subset Q_1\ni A_1^{-1}x}}{\sup} \; \underset{\xi
\in Q_1\cup Q_2 }{\sup\text{ess}}|T_\alpha(f\chi_{3(Q_0^{1}\cup
Q_0^{2})\setminus 3(Q_1\cup Q_2)})(\xi)|,
\end{align*}
where the supremum is taking over all $Q_i$ cubes in $Q_0^{i}$ for
$i=1,2$.

\begin{lema}\label{lemaalpha}
Let $0<\alpha<n$ and %For each $1\leq i \leq 2$,
 $0<\alpha_1,\alpha_2<n$ such that $\alpha_1+\alpha_2=n-\alpha$. Let $1<r_1,r_2\leq \infty$ and $s\geq 1$ such that $\frac1{r_1}+\frac1{r_2}+\frac1{s}=1$. For each $1\leq i \leq 2$, let $A_i$ be invertible
matrices satisfying hypothesis (H) and  $k_i\in S_{n-\alpha_i,r_i}\cap
H_{n-\alpha_i,r_i}$. The following estimates hold:
\begin{enumerate}[(i)]
\item for a.e. $A_i^{-1}x\in Q_0^{i}$
$$|T_{\alpha}(f\chi_{3(Q_0^{1}\cup Q_0^{2})})(x)|\leq M_{T_{\alpha},Q_0^{1}\cup Q_0^{2}}f(x),$$
\item for all $x\in \R$
$$M_{T_{\alpha}}(f)(x)\lesssim \sum_{i=1}^m M_{\alpha,s}(A_i^{-1}x) + |T_{\alpha}(f)(x)|.$$
\end{enumerate}
Therefore, %Futhermore, if $\phi \in X_{\alpha}$
$$|\{x\in\R : M_{T_{\alpha}}(f)(x)>\lambda\}|^{\frac{n-\alpha s}{n}}\leq c^s\int_{\{x\in Q : f(x)\geq \lambda|Q|^{\alpha/n}/c\}} \left(\frac{|f(x)|}{\lambda|Q|^{\alpha/n}}\right)^s dx.$$
\end{lema}

For the case $\alpha=0$, we have the following lemma:

\begin{lema}\label{lemasparse2}
 %For each $1\leq i \leq 2$,
Let $0<\alpha_1,\alpha_2<n$ such that $\alpha_1+\alpha_2=n$. Let $1<r_1,r_2\leq \infty$, $s\geq 1$ such that $\frac1{r_1}+\frac1{r_2}+\frac1{s}=1$. For each $1\leq i \leq 2$, let $A_i$ be invertible matrices satisfying hypothesis (H) and   $k_i\in S_{n-\alpha_i,r_i}\cap
H_{n-\alpha_i,r_i}$. Let  $T$ be strong type  $(p_0,p_0)$,  $1<p_0<\infty$. The following estimates hold:
\begin{enumerate}[(i)]
\item for a.e. $A_i^{-1}x\in Q_0^{i}$
$$|T(f\chi_{3(Q_0^{1}\cup Q_0^{2})})(x)|\leq \|T\|_{L^1 \rightarrow L^{1,\infty}}\sum_{i=1}^m |f(A_i^{-1}x)|+M_{T,Q_0^{1}\cup Q_0^{2}}f(x),$$
\item for all $x\in \R$
$$M_{T}(f)(x) \lesssim \sum_{i=1}^m \left[M_{s}f(A_i^{-1}x) + \|T\|_{L^1 \rightarrow L^{1,\infty}}Mf(A_i^{-1}x) \right]+ M_{\delta}(Tf)(x).$$
\end{enumerate}
Therefore
$$|\{x\in\R : M_{T}(f)(x)>\lambda\}|\leq c^s\int_{\R}\left( \frac{|f(x)|}{\lambda}\right)^sdx.$$
\end{lema}

The following lemma is the so called $3^n$ dyadic lattices trick.
This result was es\-ta\-bli\-shed in \cite{LN18} and affirms:
\begin{lema}\label{lemaDiadic}\cite{LN18}
Given a dyadic family $\mathcal{D}$ there exist $3^n$ dyadic families $\mathcal{D}_j$ such that
$$\{3Q:Q\in \mathcal{D}\}=\bigcup_{j=1}^{3^n}\mathcal{D}_j,$$
and for every cube $Q\in \mathcal{D}$ we can find a cube $R_Q$ in each $\mathcal{D}_j$ such that $Q\subset R_Q$ and $3l_Q=l_{R_Q}$.
\end{lema}

\begin{proof}[Proof of Theorem \ref{Sparse}]
We follow the ideas as in \cite{AMPRR17,IFRR17,L16} for the
domination, and adapt these in to our operator.

We claim that for any $Q_0^1,\dots,Q_0^m$, there exist
$\frac1{2}$-sparse families $\mathcal{F}_i\subset \mathcal{D}(Q_0^i)$, $i=1,\dots,m$, such that
for a.e. $x\in \bigcup_{i=1}^m A_iQ_0^i$
\begin{equation}\label{claimSparse}
|T_{\alpha,m}(f\chi_{3\bigcup_{i=1}^m Q_0^i})(x)|\leq c \sum_{i=1}^m
\sum_{Q\in \mathcal{F}_i}
|3Q|^{\alpha/n}\|f\|_{\phi,3Q}\chi_Q(A_i^{-1}x).
\end{equation}

 Suppose that we have already proved \eqref{claimSparse}. Let $\D$ be a family of dyadic cubes such that there exists $Q_0\in \D$ and supp$f\subset Q_0$. Let us take cubes $Q_j$ such that supp$f\subset 3Q_j$. We start with the $Q_0$ and let us cover $3Q_0\setminus Q_0$ by $3^n-1$ congruent cubes. Each of them satisfies $Q_0\subset 3Q_j$. We do the same for $9Q_0\setminus 3Q_0$ and so on. The union of all of those cubes, including $Q_0$, will satisfy the desired properties.

We apply the claim to each cube $Q_j$, in the following way: let $Q_0^1=\dots=Q_0^m=Q_j$ then there exists a $\frac1{2}$-sparse family $\mathcal{F}_j\subset \mathcal{D}(Q_j)\subset \D$ such that for a.e. $x\in \bigcup_{i=1}^m A_iQ_0^i$,
$$ |T_{\alpha,m}(f\chi_{3 Q_j})(x)|\chi_{\bigcup_{i=1}^m A_iQ_j}(x)\leq
c \sum_{i=1}^m \sum_{Q\in \mathcal{F}_j} |3Q|^{\alpha/n}\|f\|_{\phi,3Q}\chi_Q(A_i^{-1}x).$$

Taking $\mathcal{F}= \bigcup \mathcal{F}_j$, this is an $\frac1{2}$-sparse family. Then, %since $A_1Q_j$ are disjoint cubes,
\begin{align*}
|T_{\alpha,m}(f)(x)|
%&=\sum_{Q_j\in \mathcal{F}}  |T_{\alpha,m}(f\chi_{3 Q_j})(x)|\chi_{A_1Q_j}(x)
%\leq \sum_{Q_j\in \mathcal{F}}  |T_{\alpha,m}(f\chi_{3 Q_j})(x)|\chi_{\bigcup_{i=1}^m A_iQ_j}(x)
 \leq c \sum_{i=1}^m \sum_{Q\in \mathcal{F}} |3Q|^{\alpha/n}\|f\|_{\phi,3Q}\chi_Q(A_i^{-1}x).
\end{align*}

If we take $R_Q\in \D_j$ such that ${3Q:Q \in \mathcal{F}}\subset \bigcup_{j=1}^{3^n}\D_j$ dyadic families, $|R_Q|\leq 3^n|3Q|$, this is posible by Lemma \ref{lemaDiadic}, let
$$\mathcal{S}_j=\{R_Q\in \D_j : Q \in \mathcal{F}\}.$$
Since $\mathcal{F}$ is a $\frac1{2}$-sparse family, then $\mathcal{S}_j$ is a $\frac1{2.9^n}$-sparse family. Then, we have that
\begin{align*}
|T_{\alpha,m}(f)(x)| \leq c \sum_{j=1}^{3^n}\sum_{i=1}^m \mathcal{A}_{\alpha, r,\mathcal{S}_j}f(A_i^{-1}x).
\end{align*}

Now to  prove \eqref{claimSparse} it is suffices to show the following recursive
estimate: for each $1\leq i \leq m$ there exists a countable family
$\{P_j^i\}_j$ of pairwise disjoint cubes in $\mathcal{D}(Q_0^i)$
such that $\sum_j P_j^i \leq \frac1{2}|Q_0^i|$ and

\begin{align*}\displaystyle
|T_{\alpha,m}(f\chi_{3\bigcup_{i=1}^m
Q_0^i})(x)|\chi_{\bigcup_{i=1}^m A_iQ_0^i}(x)&\leq c \sum_{i=1}^m
\sum_{Q\in \mathcal{F}_i}
|3Q|^{\alpha/n}\|f\|_{\phi,3Q}\chi_Q(A_i^{-1}x)
\\ & \qquad
 +|T_{\alpha,m}(f\chi_{3\bigcup_{i=1}^m P_j^i})(x)|\chi_{\bigcup_{i=1}^m A_iP_j^i}(x),
\end{align*}
for a.e. $x\in \bigcup_{i=1}^m A_iQ_0^i$. 
Iterating this estimates we obtain \eqref{claimSparse} with $\mathcal{F}_i=\{P_j^{i,k}\} $ where $\{P_j^{i,0}\}=\{Q_0^i\}$, $\{P_j^{i,0}\}=\{P_j^i\}$  and $\{P_j^{i,k}\}$ are the cubes obtained at the $k$.th stage of the iterative process. Each family $\mathcal{F}_i$ is a $\frac1{2}$-sparse. Indeed, for each $P_j^{i,k}$ it suffices to choose
$$E_{P_j^{i,k}}=P_j^{i,k}\setminus \bigcup_j P_j^{i,k+1}.$$

Observe that for any
family $\{P_j^i\}_j \subset \mathcal{D}(Q_0^i)$ of  pairwise
disjoint cubes, we have
\begin{align*}
&|T_{\alpha,m}(f\chi_{3\bigcup_{i=1}^m Q_0^i})(x)|\chi_{\bigcup_{i=1}^m A_iQ_0^i}(x)
\\&\leq  |T_{\alpha,m}(f\chi_{3\bigcup_{i=1}^m Q_0^i})(x)|\chi_{\bigcup_{i=1}^m A_i (Q_0^i \setminus \cup_j P_j^i)}(x) + \sum_j|T_{\alpha,m}(f\chi_{3\bigcup_{i=1}^m Q_0^i})(x)|\chi_{ \bigcup_{i=1}^mA_iP_j^i}(x)
\\&\leq |T_{\alpha,m}(f\chi_{3\bigcup_{i=1}^m Q_0^i})(x)|\chi_{\bigcup_{i=1}^m A_i (Q_0^i \setminus \cup_j P_j^i)}(x)
+ \sum_j|T_{\alpha,m}(f\chi_{\bigcup_{i=1}^m 3(Q_0^i\setminus P_j^i)})(x)|\chi_{ \bigcup_{i=1}^mA_iP_j^i}(x)
\\&\qquad+\sum_j|T_{\alpha,m}(f\chi_{3\bigcup_{i=1}^m P_j^i})(x)|\chi_{ \bigcup_{i=1}^mA_iP_j^i}(x),
\end{align*}
for almost every $x\in \mathbb{R}^{n}$. So it is suffices to show
that we can choose a countable family  $\{P_j^i\}_j$ of pairwise
disjoint cubes in $\mathcal{D}(Q_0^i)$ such that $\sum_j P_j^i \leq
\frac1{2}|Q_0^i|$ and for a.e. $x\in \bigcup_{i=1}^m A_iQ_0^i$,
\begin{align}\label{eq: AcotSparse}
|T_{\alpha,m}(f\chi_{3\bigcup_{i=1}^m Q_0^i})(x)|&\chi_{\bigcup_{i=1}^m A_i (Q_0^i \setminus \cup_j P_j^i)}(x)
+ \sum_j|T_{\alpha,m}(f\chi_{\bigcup_{i=1}^m 3(Q_0^i\setminus P_j^i)})(x)|\chi_{ \bigcup_{i=1}^mA_iP_j^i}(x)
\nonumber
\\&\leq c  \sum_{i=1}^m\sum_{Q\in \mathcal{F}_i} |3Q|^{\alpha/n}\|f\|_{\phi,3Q}\chi_Q(A_i^{-1}x).
\end{align}
To prove this we follow the ideas  in
\cite{AMPRR17,IFRR17,L16}, with $E=\bigcup_{i=1}^mE_{\alpha}^i$ defined by,
$$E_0^i=\{x\in Q_0^i : |f|>\gamma_n \|f\|_{s,3Q_0^i}\}\cup \{x \in Q_0^i: M_{T_{0,m},\cup_i Q_0^i}(f)>\gamma_n c \sum_{i=1}^m \|f\|_{s,3Q_0^i}\},$$
 if
$\alpha=0$ and,
$$E_{\alpha}^i= \{x \in Q_0^i: M_{T_{\alpha,m},\cup_i Q_0^i}(f)>\gamma_n c \sum_{i=1}^m|3Q_0^i|^{\alpha/n}\|f\|_{s,3Q_0^i}\},$$
if $\alpha>0$.
%{\red
%Then by Lemma \ref{lemaalpha} and \ref{lemasparse2}, we can choose $\gamma_n$ such that $|E|\leq \frac1{2^{n+2}}\sum_{i=1}^m|Q_0^i|$. Now, the proof is analogouos to the
%preceeding works. }

Now, we prove that there  exists $\gamma_n$ such that $\displaystyle |E^i_\alpha|\leq \frac1{2^{n+2}}\sum_{i=1}^m|Q_0^i|$.

If $\alpha=0$, using Lemma  \ref{lemasparse2}, we have that % that  $M_{T_{0,m}}$ is weak-type  $(s,1)$, then
\begin{align*}
|E_{0}^i|&\leq  \frac{\int_{Q_0^i}|f(x)|dx}{\gamma_n\|f\|_{s,3Q_0^i}} +c \int_{\cup_{i=1}^m 3Q_0^i} \left( \frac{|f(x)|}{ \gamma_n c \sum_{i=1}^m \|f\|_{s,3Q_0^i}}\right)^s dx
\\&\leq |3Q_0^i|\frac{\frac1{|3Q_0^i|}\int_{3Q_0^i}|f(x)|dx}{\gamma_n\|f\|_{s,3Q_0^i}} +c \sum_{i=1}^m\frac1{\gamma_n^s c^s}\int_{ 3Q_0^i} \left( \frac{|f(x)|}{ \sum_{i=1}^m \|f\|_{s,3Q_0^i}}\right)^s dx
\\&\leq \frac{|3Q_0^i|}{\gamma_n}+ c\sum_{i=1}^m \frac{|3Q_0^i|}{\gamma_n^s c^s}\frac1{|3Q_0^i|}\int_{ 3Q_0^i} \left( \frac{|f(x)|}{ \|f\|_{s,3Q_0^i}}\right)^s dx
\\&= \frac{|3Q_0^i|}{\gamma_n}+ c\sum_{i=1}^m \frac{|3Q_0^i|}{\gamma_n^s c^s}
\\&\leq \left( \frac1{\gamma_n} + \frac{c^{1-s}}{\gamma_n^s}\right)\sum_{i=1}^m|3Q_0^i|.
\end{align*}
Thus, we can choose  $\gamma_n$ such that
$$m\left( \frac1{\gamma_n} + \frac{c^{1-s}}{\gamma_n^s}\right)\leq \frac1{2^{n+2}}.$$

In te case of  $\alpha>0$, by  Lemma  \ref{lemaalpha}, we have %that  $M_{T_{\alpha,m}}$ is ewatk-type $(s,\frac{n}{n-\alpha s})$, then
\begin{align*}
|E_{\alpha}^i|^{\frac{n-\alpha s}{n}}&
\leq  c_1 \int_{3\cup Q_0^i} \left(\frac{|f(x)|}{\gamma_n c_2|3Q_0^i|^{\alpha/n}\|f\|_{s,3Q_0^i}}\right)^sdx
\\ & \leq C \sum_{i=1}^m\frac1{\gamma_n^s  |3Q_0^i|^{s\alpha/n}}\int_{ 3Q_0^i} \left( \frac{|f(x)|}{ \sum_{i=1}^m \|f\|_{s,3Q_0^i}}\right)^s dx
\\&\leq C \sum_{i=1}^m \frac{|3Q_0^i|}{\gamma_n^s |3Q_0^i|^{s\alpha/n}}\frac1{|3Q_0^i|}\int_{ 3Q_0^i} \left( \frac{|f(x)|}{ \|f\|_{s,3Q_0^i}}\right)^s dx
\\&= \frac{C}{\gamma_n^s} \sum_{i=1}^m |3Q_0^i|^{1-s\alpha/n}.
\end{align*}
Thus,
$$|E_{\alpha}^i| \leq \frac{C3 ^n}{\gamma_n^{\frac{sn}{n-\alpha s}}} \sum_{i=1}^m |Q_0^i|,$$
then  it is  enough to take  $\gamma_n$ such that
$$mC3 ^n\gamma_n^{-\frac{sn}{n-\alpha s}}\leq \frac1{2^{n+2}}.$$

Now we apply Calder\'on-Zygmund decomposition to the function $\chi_{E_{\alpha}^i}$ on  $Q_0^i$ at height  $\lambda =\frac1{2^{n+1}}$. We obtain pairwise disjoint  cubes $P_j^i\in \mathcal{D}(Q_0^i)$ such that
$$\chi_{E_{\alpha}^i}(x)\leq \frac1{2^{n+1}}\qquad a.e. \;x \not\in \cup P_j^i.$$
Also we have $|E_{\alpha}^i\setminus \cup_j P_j^i|=0$,
$$\sum_j |P_j^i|=\bigg| \bigcup_j P_j^i\bigg|\leq 2^{n+1}|E_{\alpha}^i| \leq 2^{n+1}|E|\leq \frac1{2}\sum_{i=1}^m|Q_0^i|$$
and
$$ \frac1{2^{n+1}} \leq \frac1{|P_j|} \int_{P_j} \chi_{E_{\alpha}^i}(x)dx =\frac{|P_j\cap E_{\alpha}^i|}{|P_j|} \leq \frac1{2}.$$
From the last estimate it follows that  $|P_j\cap (E_{\alpha}^i)^c|>0$. Indeed,
$$|P_j^i|=|P_j^i\cap (E_{\alpha}^i)|+|P_j^i\cap (E_{\alpha}^i)^c|\leq \frac1{2}|P_j^i|+|P_j^i\cap (E_{\alpha}^i)^c|,$$
Then  $\frac1{2}|P_j^i|<|P_j^i\cap (E_{\alpha}^i)^c|$.

For $i=1,\dots, m$, observe that since $P_j^i\cap (E_{\alpha}^i)^c \not= \emptyset$, $M_{T,\cup_ iQ_0^i}(f)(x)\leq \gamma_n c \sum_{i=1}^m \|f\|_{s,3Q_0^i}$ for some  $x\in  A_iP_j^i$ and then
$$\underset{\xi \in P_j^i}{\text{ess sup}} \bigg| T_{\alpha,m}(f\chi_{\cup_i 3(Q_0^i\setminus P_j^i)})(\xi) \bigg|\leq \gamma_n c \|f\|_{s,3Q_0^i}.$$
Thus,
$$\underset{\xi \in \cup_{i=1}^m P_j^i}{\text{ess sup}} \bigg| T_{\alpha,m}(f\chi_{\cup_i 3(Q_0^i\setminus P_j^i)})(\xi) \bigg|\leq \gamma_n c \sum_{i=1}^m \|f\|_{s,3Q_0^i},$$
which allows us to control the summation in \eqref{eq: AcotSparse}. 

On the other hand, if  $\alpha=0$, by  $(i)$ in Lemma \ref{lemasparse2} we know that a.e $x\in A_iQ_0^{i}$
$$|T_{0,m}(f\chi_{3(\cup_i Q_0^i)})(x)|\leq \|T_{0,m}\|_{L^1 \rightarrow L^{1,\infty}}\sum_{i=1}^m |f(A_i^{-1}x)|+M_{T_{0,m},\cup_iQ_0^i}f(x).$$
If $x\in Q_0^i\setminus \cup_j P_j^i$ then since $|E_{0}^i\setminus \cup_j P_j^i|=0$ we have that,  by the definition of $E^i_0$, 
$$|f(x)|\leq \gamma_n \|f\|_{s,3Q_0^i} \leq  \gamma_n \sum_{i=1}^m \|f\|_{s,3Q_0^i},$$
a.e. $x$ and also that  $M_{T_{0,m},\cup_{i=1}^m Q_0^i}f(x)\leq \gamma_n \sum_{i=1}^m \|f\|_{s,3Q_0^i}$, a.e. $x$. Consequently
$$ \bigg| T_{\alpha,m}(f\chi_{3\cup_iQ_0^i})(x) \bigg|\leq \gamma_n c\sum_{i=1}^m \|f\|_{s,3Q_0^i}.$$
Those estimates allows us to control the remaing terms in \eqref{eq: AcotSparse} for $\alpha=0$.

If  $\alpha>0$, by $(i)$  in Lemma  \ref{lemaalpha} we know that a.e.  $A_i^{-1}x\in Q_0^{i}$
$$|T_{\alpha,m}(f\chi_{3\cup_i Q_0^{i}})(x)|\leq M_{T_{\alpha,m},\cup_i Q_0^{i}}f(x),$$
then proceding as above, we prove \eqref{eq: AcotSparse} for $\alpha>0$.
%{\vb luego procedemos analogamente al caso anterior. }

\end{proof}

\subsection{Proof of previuos lemma}
In this subsection we prove the Lemmas \ref{lemaalpha} and
\ref{lemasparse2}.

\begin{proof}[Proof of Lemma \ref{lemaalpha}]
 ($i$): %was established in \cite{IFRV18}, so we only have to prove part
For $i=1,2$, let $\tilde{Q}^i$ be a cube centered at $A_i^{-1}x$
with lenght $t$ such that $\tilde{Q}^i \subset Q_0^i$.
\begin{align*}
|T_{\alpha}(f\chi_{3(Q_0^{1}\cup Q_0^{2})})(x)|\leq
|T_{\alpha}(f\chi_{3(\tilde{Q}^1\cup
\tilde{Q}^2)})(x)|+|T_{\alpha}(f\chi_{3(Q_0^{1}\cup Q_0^{2})
\setminus 3(\tilde{Q}^1\cup \tilde{Q}^2))})(x)|.
\end{align*}

Let $B^i$ be a ball with centred at $A_i^{-1}x$ with radius $R=\frac 32
\sqrt{n}t$, then $3\tilde{Q}^i \subset B^{i}$,
\begin{align*}
|T_{\alpha}&(f\chi_{3(\tilde{Q}^1\cup \tilde{Q}^2)})(x)|\leq
|T_{\alpha}(f\chi_{B^1\cup B^2})(x)| \leq
|T_{\alpha}(f\chi_{B^1})(x)|+|T_{\alpha}(f\chi_{B^2})(x)|.
%\leq \int_{B^1\cup B^2} |k_1(x-A_1z)||k_2(x-A_2z)||f(z)|dz
%\\&\leq \int_{B^1} |k_1(x-A_1z)||k_2(x-A_2z)||f(z)|dz +\int_{B^2} |k_1(x-A_1z)||k_2(x-A_2z)||f(z)|dz
\end{align*}
For $B^1$, we consider the sets $$X^1=B^1\cap\{z : |x-A_1z|\leq
|x-A_2z|\},$$ and $X^2=B^1\setminus X^1$.

For $X^1$, we decompose the set in the following  way
$$C^1_j=\{z : |x-A_1z|\sim 2^{-j}R\|A_1\|\},$$
where $\|A_1\|=\underset{x\not = 0}{\sup}
\frac{|A_1x|}{|x|}$. Observe that
$\displaystyle X^1\subset \bigcup_{j=1}^{\infty} C^1_j$. Let
$\tilde{B}_j=A_1^{-1}B(x,2^{-j}R\|A_1\|)$. Since $k_1 \in
S_{n-\alpha_1,r_1}$ and $k_2 \in S_{n-\alpha_2,r_2}$, we have
\begin{align*}
\int_{X^1}&|k_1(x-A_1z)||k_2(x-A_2z)||f(z)|dz
\\&\leq \sum_{j=1}^{\infty}\frac{|\tilde{B}_{j+1}|}{|\tilde{B}_{j+1}|} \int_{C^1_{j+1}}|k_1(x-A_1z)||k_2(x-A_2z)||f(z)|dz
\\&\leq \sum_{j=1}^{\infty} |\tilde{B}_{j+1}| \|k_1(x-A_1\cdot)\chi_{C^1_{j+1}}\|_{r_1,\tilde{B}_{j+1}}\|k_2(x-A_2\cdot)\chi_{C^1_{j+1}}\|_{r_2,\tilde{B}_{j+1}}\|f\|_{s,\tilde{B}_{j+1}}
\\&\leq M_{s}(f)(A_1^{-1}x) \sum_{j=1}^{\infty} (2^{-j}R\|A_1\|)^{n}(2^{-j}R\|A_1\|)^{\alpha-n}
=cM_{s}(f)(A_1^{-1}x) R^{\alpha}\sum_{j=1}^{\infty} 2^{-j\alpha}
\\&=cM_{s}(f)(A_1^{-1}x) R^{\alpha}.
\end{align*}

In an analogous way we obtain
$$\int_{X^2}|k_1(x-A_1z)||k_2(x-A_2z)||f(z)|dz \leq cM_{s}(f)(A_2^{-1}x) R^{\alpha}.$$

Hence, we have for all $R>0$,
$$|T_{\alpha}(f\chi_{3(\tilde{Q}^1\cup \tilde{Q}^2)})(x)|\leq c\left(M_{s}(f)(A_1^{-1}x) +M_{s}(f)(A_2^{-1}x)\right) R^{\alpha} + M_{T_{\alpha},Q_0^{1}\cup Q_0^{2}}f(x). $$

Taking $t\rightarrow 0$, we have $R\rightarrow 0$, then
$$|T_{\alpha}(f\chi_{3(\tilde{Q}^1\cup \tilde{Q}^2)})(x)|\leq M_{T_{\alpha},Q_0^{1}\cup Q_0^{2}}f(x).$$

($ii$): We are going to follow ideas in \cite{IFRR17} , \cite{IFRi18}
and \cite{IFRV18}. % We only prove the case $m=2$, the general case is analogous.
%Let $x,x',\xi in Q\subset \frac1{2} 3Q$. Then
Let $x\in \mathbb{R}^n$, for $i=1,2$ let $Q_i$ be a cube containing $A_i^{-1}x$. Let $B_x$ be a ball %with radius $R$ and center $x$
  such that $3(Q_1\cup Q_2)\subset B_x$.
For every $\xi \in Q_1\cup Q_2$, we have
\begin{align}\label{Tlejos}
|T_{\alpha}(f\chi_{\R\setminus 3(Q_1\cup Q_2)})(\xi)|&\leq
|T_{\alpha}(f\chi_{\R\setminus
B_x})(\xi)-T_\alpha(f\chi_{\R\setminus B_x})(x)| \nonumber
\\&\qquad+|T_{\alpha}(f\chi_{B_x \setminus 3(Q_1\cup Q_2)})(\xi)|
+ |T_{\alpha}(f\chi_{\R\setminus B_x})(x)| \nonumber
\\&\lesssim  |T_{\alpha}(f\chi_{\R\setminus B_x})(\xi)-T_{\alpha}(f\chi_{\R\setminus B_x})(x)| \nonumber
\\&\qquad+ |T_{\alpha}(f\chi_{B_x \setminus 3(Q_1\cup Q_2)})(\xi)|+
T_{\alpha}(|f|)(x).
\end{align}

Let $Z^1=B_x^c \cap \{z : |x-A_1z|\leq |x-A_2z|\}$ and $Z^2$
define in analogous way. Then,
\begin{align*}
|T_{\alpha}(f\chi_{\R\setminus
B_x})(\xi)&-T_\alpha(f\chi_{\R\setminus B_x})(x)|\leq
\int_{B_x^c}|K(\xi,y)-K(x,y)||f(y)|dy
\\&\leq \int_{Z^1}|K(\xi,y)-K(x,y)||f(y)|dy + \int_{Z^2}|K(\xi,y)-K(x,y)||f(y)|dy.
\end{align*}

 Let $R=\frac{l(Q)}{2\|A_1^{-1}\|}$ where
$\|A_1^{-1}\|=\underset{x\not =
	0}{\sup}\frac{|A_1^{-1}x|}{|x|}$. For $j\in \mathbb{N}$, let define 
$$D_j^1=\{z\in Z^1: |x-A_1z|\sim 2^{j+1}R\}.$$
Observe that $A_1^{-1}B(x,2R)\subset 3Q_1$,
  $$D_j^1\subset \{z:|x-A_1z|\sim 2^{j+1}R\}\subset A_1^{-1}B(x,2^{j+2}R)=:\tilde{B}_{1,j},$$
   and that  $Z^1\subset \bigcup_{j\in \mathbb{N}}D_j^1$.
%Since  $k_1\in H_{n-\alpha_1,\Phi_1}$ and $k_2\in S_{n-\alpha_2,\Phi_2}$,

If $K(x,y)=k_1(x-A_1y)k_2(x-A_2y)$, we have
\begin{align*}
|K(\xi,y)-K(x,y)|\leq &|k_1(\xi-A_1y)-k_1(x-A_1y)||k_2(\xi-A_2y)|
\\&+|k_1(x-A_1y)||k_2(\xi-A_2y)-k_2(x-A_2y)|.
\end{align*}
Then,
\begin{align*}
\int_{Z^1}&|k_1(\xi-A_1y)-k_1(x-A_1y)||k_2(\xi-A_2y)||f(y)|dy
\\& \leq \sum_{j=1}^{\infty}\frac{|\tilde{B}_{1,j}|}{|\tilde{B}_{1,j}|}\int_{\tilde{B}_{1,j}}\chi_{D_j^1}(y)|k_1(\xi-A_1y)-k_1(x-A_1y)||k_2(\xi-A_2y)||f(y)|dy
\\&\leq \sum_{j=1}^{\infty}|\tilde{B}_{1,j}| \|(k_1(\xi-A_1\cdot)-k_1(x-A_1\cdot))\chi_{D_j^1}\|_{r_1,\tilde{B}_{1,j}}\|k_2(\xi-A_2\cdot)\chi_{D_j^1}\|_{r_2,\tilde{B}_{1,j}}\|f\|_{s,\tilde{B}_{1,j}}
\\&\leq M_{\alpha,s}f(A_1^{-1}x)\sum_{j=1}^{\infty}|\tilde{B}_{1,j}|^{1-\alpha} \|(k_1(\xi-A_1\cdot)-k_1(x-A_1\cdot))\chi_{D_j^1}\|_{r_1,\tilde{B}_{1,j}}\|k_2(\xi-A_2\cdot)\chi_{D_j^1}\|_{r_2,\tilde{B}_{1,j}}.
\end{align*}
If $z\in D_j^1$ then $|x-A_2z|\geq |x-A_1z| \geq 2^{j+1}R$. So we
write $D_j^1=\bigcup_{k\geq j}(D_j^1)_{k,2}$ where
$$(D_j^1)_{k,2}=\{z\in D_j^1 : |x-A_2z|\sim 2^{k+1}R \}.$$

Since $k_2\in S_{n-\alpha_2,r_2}$, we have
$$\|k_2(\xi-A_2\cdot)\chi_{D_j^1}\|_{r_2,\tilde{B}_{1,j}}\lesssim (2^jR)^{-\alpha_2}. $$

Since  $k_1\in H_{n-\alpha_1,r_1}$ and $\alpha_1 + \alpha_2 = n-\alpha$,
\begin{align*}
\int_{Z^1}&|k_1(\xi-A_1y)-k_1(x-A_1y)||k_2(\xi-A_2y)||f(y)|dy
\\&\lesssim {M_{\alpha,s}}f(A_1^{-1}x)\sum_{j=1}^{\infty}|\tilde{B}_{1,j}|^{1-\alpha/n-\alpha_2/n} \|(k_1(\xi-A_1\cdot)-k_1(x-A_1\cdot))\chi_{D_j^1}\|_{r_1,\tilde{B}_{1,j}}
\\&\lesssim M_{\alpha,s}f(A_1^{-1}x).
\end{align*}
In an analogous way we have
$$\int_{Z^1}|k_1(x-A_1y)||k_2(\xi-A_2y)-k_2(x-A_2y)||f(y)|dy \lesssim M_{\alpha,s}f(A_1^{-1}x).$$
%{\red and the intregral in $Z^2$.}
Hence we obtain
\begin{equation}\label{TfueraBx}
|T_{\alpha}(f\chi_{\R\setminus
B_x})(\xi)-T_\alpha(f\chi_{\R\setminus B_x})(x)|\lesssim
M_{\alpha,s}f(A_1^{-1}x)+M_{\alpha,s}f(A_2^{-1}x).
\end{equation}

For the second term of \eqref{Tlejos}, we have
\begin{align*}
|T_{\alpha}(f\chi_{B_x \setminus 3(Q_1\cup Q_2)})(\xi)|&\leq
\int_{B_x \setminus 3(Q_1\cup
Q_2)}|k_1(\xi-A_1y)||k_2(\xi-A_2y)||f(y)|dy
\\&\leq \int_{Y^1} |k_1(\xi-A_1y)||k_2(\xi-A_2y)||f(y)|dy
\\ &\quad+ \int_{Y^2}|k_1(\xi-A_1y)||k_2(\xi-A_2y)||f(y)|dy,
\end{align*}
where
$$Y^1=\left(B_x \setminus 3(Q_1\cup Q_2)\right) \cap \{z : |x-A_1z|\leq |x-A_2z|\},$$
and $Y^2=\R \setminus Y^1$. Observe that for $i=1,2$,  $Y^i \subset
A_i^{-1}B(x,2^lR)\setminus  A_i^{-1}B(x,2R)$ for some $l\in
\mathbb{N}$ and let $B_{j}^i:=A_i^{-1}B(x,2^jR)$. Then, by H\"older
inequality we obtain
\begin{align*}
\int_{Y^1} &|k_1(\xi-A_1y)||k_2(\xi-A_2y)||f(y)|dy
\\&\leq \sum_{j=1}^{l-1} \frac{|B_{j+1}^1|}{|B_{j+1}^1|} \int_{B_{j+1}^1\setminus B_{j}^1} |k_1(\xi-A_1y)||k_2(\xi-A_2y)||f(y)|dy
\\&\leq \sum_{j=1}^{l-1} |B_{j+1}^1| \|k_1(\xi-A_1\cdot)\chi_{B_{j+1}^1\setminus B_{j}^1} \|_{r_1,B_{j+1}^1} \|k_2(\xi-A_2\cdot)\chi_{B_{j+1}^1\setminus B_{j}^1} \|_{r_2,B_{j+1}^1}\|f\|_{s,B_{j+1}^1}.
%\\&\leq M_{\alpha,\phi}f(A_1^{-1}x)\sum_{j=1}^{l-1} |B_{j+1}^1|^{1-\alpha/n} \|k_1(\xi-A_1\cdot)\chi_{B_{j+1}^1\setminus B_{j}^1} \|_{\Psi_1,B_{j+1}^1} \|k_2(\xi-A_2\cdot)\chi_{B_{j+1}^1\setminus B_{j}^1} \|_{\Psi_2,B_{j+1}^1}
\end{align*}
Since $k_2\in S_{n-\alpha_2,r_2}$, as above we have
$$\|k_2(\xi-A_2\cdot)\chi_{B_{j+1}^1\setminus B_{j}^1} \|_{r_2,B_{j+1}^1}\lesssim (2^{j+1}R)^{-\alpha_2}.$$

Then, since $k_1\in S_{n-\alpha_1,r_1}$ and  $\alpha_1 + \alpha_2
= n-\alpha$,  we get

\begin{align*}
\int_{Y^1} &|k_1(\xi-A_1y)||k_2(\xi-A_2y)||f(y)|dy
\\&\leq c M_{\alpha,s}f(A_1^{-1}x)\sum_{j=1}^{l-1} |B_{j+1}^1|^{1-\alpha/n-\alpha_2/n} \|k_1(\xi-A_1\cdot)\chi_{B_{j+1}^1\setminus B_{j}^1} \|_{r_1,B_{j+1}^1}
\\&\leq c M_{\alpha,s}f(A_1^{-1}x).
\end{align*}
In an analogous  way,
$$\int_{Y^2} |k_1(\xi-A_1y)||k_2(\xi-A_2y)||f(y)|dy\lesssim  M_{\alpha,s}f(A_2^{-1}x).$$
By \eqref{Tlejos}, \eqref{TfueraBx} and the last inequalities, we
obtain
\begin{align*}
|T_{\alpha}(f\chi_{\R\setminus 3(Q_1\cup Q_2)})(\xi)|& \lesssim
M_{\alpha,s}f(A_1^{-1}x) +
M_{\alpha,s}f(A_2^{-1}x)+T_{\alpha}(|f|)(x).
\end{align*}
%Since $M_{\alpha,s}$ is bounded from $L^s$ into $L^{\frac{ns}{n-\alpha s}}$ and the estimate of $T_{\alpha}(|f|)$ follows
%from the Coifman-Fefferman inequality in \cite{IFRi18} and the
%estimate of $M_{\alpha,s}$. Then we obtain the desired inequality.
By the Coifman-Fefferman inequality in \cite{IFRi18} for $T_{\alpha}(|f|)$ and using that $M_{\alpha,s}$ is bounded from $L^s$ in $L^{\frac{ns}{n-\alpha s},\infty}$  we obtain the desired inequality.
\end{proof}

\begin{proof}[Proof of Lemma \ref{lemasparse2}] We follows the ideas in \cite{L16}. We only give the changes in the proof.\\

$(i)$ For $i=1,\dots,m$, let $A_i^{-1}x\in \text{int}Q_0^i$ and
suppose that $A_i^{-1}x$ is a point of appro\-ximate continuity of
$T(f\chi_{3Q_0^i})$ (see \cite{EG15}). Then for every $\epsilon >0$, let $t>0$
$$E_t^i=\{y \in B(A_i^{-1}x,t) : |T(f\chi_{3Q_0^i})(y)-T(f\chi_{3Q_0^i})(x)|<\epsilon/m\}$$
Let $Q(A_i^{-1}x,t)$ the smallest cube centreded in $A_i^{-1}x$ containing $B(A_i^{-1}x,t)$. Take $t$ such that  $Q(A_i^{-1}x,t)\subset Q_0^i$.\\
Then for a.e. $y\in \bigcup_{i=1}^m E_t^i$ we have
$$|T(f\chi_{3(Q_0^{1}\cup Q_0^{2})})(x)|\leq \|T\|_{L^1 \rightarrow L^{1,\infty}}\sum_{i=1}^m |f(A_i^{-1}x)|+M_{T,Q_0^{1}\cup Q_0^{2}}f(x).$$\\

$(ii)$ Let $Q_1,\dots,Q_m$ cubes such that $A_i^{-1}x\in Q_i$ and
let $\xi \in \bigcup_{i=1}^m Q_i$. Then

\begin{align*}
|T(f\chi_{\mathbb{R}^n \setminus 3\bigcup_{i=1}^m Q_i})(\xi)| &\leq
|T(f\chi_{\mathbb{R}^n \setminus 3\bigcup_{i=1}^m Q_i})(\xi) -
T(f\chi_{\mathbb{R}^n \setminus 3\bigcup_{i=1}^m Q_i})(x')|
\\ &\qquad+ |Tf(x')| + |T(f\chi_{3\bigcup_{i=1}^m Q_i})(x')|
\\&\leq |T(f\chi_{\mathbb{R}^n \setminus 3\bigcup_{i=1}^m Q_i})(\xi) - T(f\chi_{\mathbb{R}^n \setminus 3\bigcup_{i=1}^m Q_i})(x')|
\\ &\qquad+ |Tf(x')| + \sum_{i=1}^m|T(f\chi_{3 Q_i})(x')|.
\end{align*}
As in the  fractional case, since $k_i\in
S_{n-\alpha_i,r_i}\cap H_{n-\alpha_i,r_i}$ for $i=1,\dots, m$,
we have
$$|T(f\chi_{\mathbb{R}^n \setminus 3\bigcup_{i=1}^m Q_i})(\xi) - T(f\chi_{\mathbb{R}^n \setminus 3\bigcup_{i=1}^m Q_i})(x')| \leq c \sum_{i=1}^{m} M_{s}f(A_i^{-1}x).$$
Now, let $Q$ be a cube such that $\bigcup_{i=1}^m Q_i\subset Q$ and
$x\in Q$. Taking average in $(L^{\delta}(Q),\frac{dx'}{|Q|})$
with $0<\delta <1$, we have
\begin{align*}
|T(f\chi_{\mathbb{R}^n \setminus 3\bigcup_{i=1}^m Q_i})(\xi)| &\leq
c \sum_{i=1}^{m} M_{s}f(A_i^{-1}x) + M_\delta(Tf)(x) +
\sum_{i=1}^m \left(\frac1{|Q|}\int_Q|T(f\chi_{3
Q_i})(x')|^{\delta}dx'\right)^{1/{\delta}}.
\end{align*} 
For the last term, by Kolmogorov's inequality we have
\begin{align*}
\left(\frac1{|Q|}\int_Q|T(f\chi_{3
Q_i})(x')|^{\delta}dx'\right)^{1/{\delta}} &\leq c \|T\|_{L^1
\rightarrow L^{1,\infty}} \frac1{|Q|}\int_Q|f\chi_{3 Q_i}(x')|dx'
\\&\leq c \|T\|_{L^1 \rightarrow L^{1,\infty}} \frac1{|Q_i|}\int_{3 Q_i}|f(x')|dx'
\\&\leq c \|T\|_{L^1 \rightarrow L^{1,\infty}} Mf(A_i^{-1}x).
%\\&\leq c \|T\|_{L^1 \rightarrow L^{1,\infty}} M_{\phi}f(A_i^{-1}x)
\end{align*}
Then, we obtain
$$M_{T}(f)(x)\lesssim \sum_{i=1}^m \left[M_{s}f(A_i^{-1}x) + \|T\|_{L^1 \rightarrow L^{1,\infty}}Mf(A_i^{-1}x) \right]+ M_{\delta}(Tf)(x).$$

Now, we prove the endpoint estimate. Observe that for $p=p_0$
\begin{align*}
\|M_{\delta}(Tf)\|_{L^{p,\infty}}&=\|M(|Tf|^{\delta})\|_{L^{p/\delta,\infty}}^{1/\delta}
\\&\leq c \||Tf|^{\delta}\|_{L^{p/\delta,\infty}}^{1/\delta} = c\|Tf\|_{L^{p,\infty}}
\\& \leq c\|T\|_{L^p\rightarrow L^{p,\infty}}\|f\|_{L^{p}}.
\end{align*}
Then, using Lemma 4.4 in \cite{IFRR17},
$$|\{x\in\R : M_{\delta}(Tf)(x)>\lambda\}|\lesssim \int_{\R}\left( c \frac{|f(x)|}{\lambda}\right)^s dx.$$
Since $M_{s}$ is bounded from $L^s$ in $L^{s,\infty}$, we obtain the desired inequality.
\end{proof}

\section{Apendix: Properties of $\mathcal{A}_{A, p,q}$}
In this section, we present some properties and remarks about this new class of weights.

\begin{prop}\label{PropwA}
Let $0\leq \alpha <n$, $1<p<\frac{n}{\alpha}$ and $\frac1{q}=\frac1{p}-\frac{\alpha}{n}$. If $w\in \mathcal{A}_{A,p,q}$, then $w(Ax)\leq [w]_{\mathcal{A}_{A,p,q}} w(x)$ a.e.$x\in
\mathbb{R}^{n}.$
\end{prop}

%We define the weight class, $\mathcal{A}_{ p,q, A}$, this is an associated class of $\mathcal{A}_{ p,q, A}$.

%\begin{itemize}
%\item  The  weight $w$ is in the class $\mathcal{A}_{A, p,q}$, $1<p\leq \infty$, $1\leq q<\infty$ if
%\begin{equation}%\label{Apq}
%[w]_{\mathcal{A}_{A,p,q}}=\underset{Q}{\sup} \left( \frac{1}{\left\vert Q\right\vert }\int\limits_{Q}w^{q}(Ax)dx\right) ^{%
%\frac{1}{q}}\left( \frac{1}{\left\vert Q\right\vert }\int\limits_{Q}w^{-p^{%
%\prime }}(x)dx\right) ^{\frac{1}{p^{\prime }}}\leq \infty.
%\end{equation}

%\item

%The  weight $w$ is in the class $\mathcal{A}_{ p,q, A}$,   $1<p\leq \infty$, $1\leq q<\infty$ if
%\begin{equation}\label{pqA}
%[w]_{\mathcal{A}_{p,q, A}}=\underset{Q}{\sup}\left( \frac{1}{\left\vert Q\right\vert }\int\limits_{Q}w^{q}(x)dx\right) ^{%
%\frac{1}{q}}\left( \frac{1}{\left\vert Q\right\vert }\int\limits_{Q}w^{-p^{%
%\prime }}(Ax)dx\right) ^{\frac{1}{p^{\prime }}}\leq \infty.
%\end{equation}
%\end{itemize}

The class $\mathcal{A}_{A,p,q}$ satisfy some properties as the Muckenhoupt weights.

\begin{prop}
Let $w$ be a weight.
\begin{enumerate}[(i)]
\item If $p<q$
$$w\in \mathcal{A}_{A,p}\Rightarrow w\in \mathcal{A}_{A,q}$$
%\item $w\in \mathcal{A}_{p,A^{-1}}$ if and only if $w\in \mathcal{A}_{A,p}$
%\item
%$$w\in \mathcal{A}_{A,p}\Rightarrow w^{1-p'}\in \mathcal{A}_{A^{-1},p'}$$
\item If $w_0 \in \mathcal{A}_{A,1}$, $w_1 \in \mathcal{A}_{A^{-1},1}$ Then $w=w_0w_1^{1-p} \in\mathcal{A}_{A,p}$
\end{enumerate}
\end{prop}

\begin{prop} Let $A$ be an invertible matrix and $w\in \mathcal{A}_{A,p}$. Then
\begin{enumerate}[(i)]
%\item $[\delta^{\lambda}(w)]_{\mathcal{A}_{A,p}}=[w]_{\mathcal{A}_{A,p}}$ where $\delta^{\lambda}(x)=w(\lambda x_1,\dots,\lambda x_n)$.
%\item $[\tau^{z}(w)]_{\mathcal{A}_{A,p}}=[w]_{\mathcal{A}_{A,p}}$ where $\tau^z(w)(x)=w(x-z), \; z\in \mathbb{R}^n$.
%\item $[\lambda w]_{\mathcal{A}_{A,p}}=[w]_{\mathcal{A}_{A,p}}$  for all $\lambda >0$.
\item \begin{align}\label{PropdetA}
|det(A)|^{-1}\underset{Q}{\sup}\left(\frac{|AQ\cap Q|}{|Q|}\right)^{p}\leq [w]_{\mathcal{A}_{A,p}}.
\end{align}
\item
$$[w]_{\mathcal{A}_{A,p}}=\underset{Q}{\sup} \underset{\underset{|f|>0 \; a.e. \;in \;Q}{f\in L^p(Q,w)}}{\sup} \frac{\left(\frac1{|Q|}\int_Q|f|\right)^p}{\frac1{w_A(Q)}\int_Q|f|^pw}$$
\item The  ``A-doubling" property: For all $\lambda \geq1$ and all $Q$ we have
$$w_A(\lambda Q)\leq \lambda^{np}[w]_{\mathcal{A}_{A,p}}w(Q)$$
where $\lambda Q$ is the cube with same center as Q and size length $\lambda$ times the side length of $Q$.
\end{enumerate}
\end{prop}

\begin{remark}
Observe that \eqref{PropdetA} does not implies that the constant must be greater than $1$.%show that the weight constant can be less that $1$ for some $A$.
\end{remark}

\begin{prop}\label{propAp}
Let $1<p<\infty$, $w$ be a weight, $\sigma=w^{-p'/p}$ and $A, B$ be invertible matrices.
\begin{enumerate}[(i)]
%\item $w\in \mathcal{A}_{A,p}$ if and only if $\sigma \in  \mathcal{A}_{p',A}$ and $[w]_{\mathcal{A}_{A,p}}=[\sigma]_{\mathcal{A}_{p',A}}$
%\vspace{0.2pt}
%\item $w\in \mathcal{A}_{A,p}$ if and only if $w\in \mathcal{A}_{p,A^{-1}}$ and $[w]_{\mathcal{A}_{A,p}}\simeq [w]_{\mathcal{A}_{p,A^{-1}}}$
%\vspace{0.2pt}
%\item If  $w\in \mathcal{A}_{A,p}\cap \mathcal{A}_{p,A}$ then $w\in \mathcal{A}_{p}$ and $[w]_{\mathcal{A}_{p}}\leq [w]_{\mathcal{A}_{A,p}}[w]_{\mathcal{A}_{p,A}}$
\item If  $w\in \mathcal{A}_{A,p}\cap \mathcal{A}_{A^{-1},p}$ then $w\in \mathcal{A}_{p}$ and $[w]_{\mathcal{A}_{p}}\leq [w]_{\mathcal{A}_{A,p}}[w]_{\mathcal{A}_{A^{-1},p}}$.
\vspace{0.2pt}
\item If  $w\in \mathcal{A}_{A,p}\cap \mathcal{A}_{B,p}$ then $w\in \mathcal{A}_{AB,p}$ and $[w]_{\mathcal{A}_{AB,p}}\lesssim [w]_{\mathcal{A}_{A,p}}[w]_{\mathcal{A}_{B,p}}$.
\vspace{0.2pt}
\item Let $Q$ be a cube. If $w\in \mathcal{A}_{A,p}$ then $\frac{1}{|Q|}\int_Q \left(w_A w^{-1}\right)^{1/p}\leq [w]_{\mathcal{A}_{A,p}}^{1/p}$.
%{\red then not necessary  the constant $[w]_{\mathcal{A}_{A,p}}\geq 1$. Example???}
\end{enumerate}
\end{prop}

The following results show in some cases a relation of this class with the Muckenhoupt class and the condition $w_A\lesssim w$, i.e. $w(Ax)\leq c w(x)$ a.e.$x\in \mathbb{R}^n$.

\begin{prop} Let $A$ be a invertible matrix and $w$ be a weight.
If $w\in A_p  {\text{ and }} w_A \lesssim  w$ then   $w\in \mathcal{A}_{A,p}$.
% Then
%\begin{enumerate}[(i)]
%\item If $w\in \mathcal{A}_{A,p}$ then  $w_A\lesssim w$
%\item If $w\in \mathcal{A}_{p,A}$ then   $w\lesssim w_A$
%\item If $w\in \mathcal{A}_{A,p}\cap \mathcal{A}_{A^{-1},p}$ then  $w\in A_p $
%\item If $w\in A_p  {\text{ and }} w_A \lesssim  w$ then   $w\in \mathcal{A}_{A,p}$
%\item If $w\in A_p  {\text{ and }} w \lesssim  w_A$ then  $w\in \mathcal{A}_{p,A}$
%\end{enumerate}
\end{prop}

\begin{teo}
Let $w$ be a weight and $A$ be an invertible matrix, the weight
$$w\in \mathcal{A}_{A,p}\cap \mathcal{A}_{A^{-1},p} \qquad {\text{ if and only if }} \qquad w\in A_p  {\text{ and }} w_A \sim w.$$
\end{teo}

\begin{remark}
Observe that we have the following inclusions
$$A_p\cap \{w: w_A \lesssim w\}\subset \{w : {\text{\it testing condition for }} M_{A^{-1}} \} \subset \mathcal{A}_{A,p}.$$
\end{remark}

Now, we present examples of matrices such that $\mathcal{A}_{A,p}$ is a subclass of $\mathcal{A}_{p}$ the Mucken\-houpt class.

\begin{corol} Let $1<p<\infty$, $w$ be a weight and $A$ be an invertible matrix.
\begin{enumerate}[(i)]
\item If $A^{-1}=A$ and  $w\in \mathcal{A}_{A,p}$ then $w\in \mathcal{A}_{p}$ and $[w]_{\mathcal{A}_{p}}\lesssim [w]_{\mathcal{A}_{A,p}}^2$.
\item If $A^N=A$ for some $N \in \mathbb{N}$ and  $w\in \mathcal{A}_{A,p}$ then $w\in \mathcal{A}_{p}$ and $[w]_{\mathcal{A}_{p}}\lesssim [w]_{\mathcal{A}_{A,p}}^N$.
\end{enumerate}
\end{corol}

 An open question if there exists a matrix $A$ such that  $\mathcal{A}_{A,p}$ is greater than $A_p\cap \{w: w_A \lesssim w\}$.
 
 \

\subsection*{Acknowledgement} We want to thank to Dra. Marta Urciuolo, for her genero\-sity and knowledge given to the group of Analysis and Differential Equations of  the FaMAF, Universidad Nacional de Córdoba.  She worked for several years, using diffe\-rent approaches, with this kind of fractional operators. In this particular case, she was the first in start trying to characterize the classes of weights $\mathcal{A}_{A,p} $ or $\mathcal{A}_{A,p,q}$, ``Marta's  weights", as we like  to call them.   
	
\bibliographystyle{acm}
\bibliography{Biblio}

\begin{thebibliography}{10}

\bibitem{AMPRR17}
{\sc Accomazzo, N., Mart{\'\i}nez-Perales, J.~C., and Rivera-R{\'\i}os, I.~P.}
\newblock On bloom type estimates for iterated commutators of fractional
  integrals.
\newblock {\em To appear in Indiana University Mathematics Journal\/} (2018).

\bibitem{EG15}
{\sc Evans, L.~C., and Gariepy, R.~F.}
\newblock {\em Measure theory and fine properties of functions}.
\newblock CRC press, 2015.

\bibitem{FH18}
{\sc Fackler, S., and Hyt{\"o}nen, T.~P.}
\newblock Off-diagonal sharp two-weight estimates for sparse operators.
\newblock {\em New York J. Math 24\/} (2018), 21--42.

\bibitem{FF15}
{\sc Ferreyra, E.~V., and Flores, G.~J.}
\newblock Weighted estimates for integral operators on local {B}{M}{O} type
  spaces.
\newblock {\em Mathematische Nachrichten 288}, 8-9 (2015), 905--916.

\bibitem{GU93}
{\sc Godoy, T., and Urciuolo, M.}
\newblock About the {L}p-boundedness of some integral operators.
\newblock {\em Revista de la Uni{\'o}n Matem{\'a}tica Argentina 38}, 3 (1993),
  192--195.

\bibitem{IFRR17}
{\sc Iba{\~n}ez-Firnkorn, G.~H., and Rivera-R{\'\i}os, I.~P.}
\newblock Sparse and weighted estimates for generalized h{\"o}rmander operators
  and commutators.
\newblock {\em Monatshefte f{\"u}r Mathematik 191}, 1 (2020), 125--173.

\bibitem{IFRi18}
{\sc Iba{\~n}ez-Firnkorn, G.~H., and Riveros, M.~S.}
\newblock Certain fractional type operators with {H}{\"o}rmander conditions.
\newblock {\em To appear in Ann. Acad. Sci. Fenn. Math.\/} (2018).

\bibitem{IFRV18}
{\sc Iba{\~n}ez-Firnkorn, G.~H., Riveros, M.~S., and Vidal, R.~E.}
\newblock Sharp bounds for fractional operator with $ {L}^{\alpha,
  r'}${-H}{\"o}rmander conditions.
\newblock {\em arXiv preprint arXiv:1804.09631\/} (2018).

\bibitem{LSUT10}
{\sc Lacey, M.~T., Sawyer, E.~T., and Uriarte-Tuero, I.}
\newblock Two weight inequalities for discrete positive operators.

\bibitem{L16}
{\sc Lerner, A.~K.}
\newblock On pointwise estimates involving sparse operators.
\newblock {\em New York J. Math 22\/} (2016), 341--349.

\bibitem{LN18}
{\sc Lerner, A.~K., and Nazarov, F.}
\newblock Intuitive dyadic calculus: the basics.
\newblock {\em Expositiones Mathematicae\/} (2018).

\bibitem{M72}
{\sc Muckenhoupt, B.}
\newblock Weighted norm inequalities for the hardy maximal function.
\newblock {\em Transactions of the American Mathematical Society 165\/} (1972),
  207--226.

\bibitem{MW74}
{\sc Muckenhoupt, B., and Wheeden, R.}
\newblock Weighted norm inequalities for fractional integrals.
\newblock {\em Transactions of the American Mathematical Society 192\/} (1974),
  261--274.

\bibitem{RS88}
{\sc Ricci, F., and Sj{\"o}gren, P.}
\newblock Two-parameter maximal functions in the {H}eisenberg group.
\newblock {\em Mathematische Zeitschrift 199}, 4 (1988), 565--575.

\bibitem{RU14}
{\sc Riveros, M., and Urciuolo, M.}
\newblock Weighted inequalities for some integral operators with rough kernels.
\newblock {\em Open Mathematics 12}, 4 (2014), 636--647.

\bibitem{RoU13}
{\sc Rocha, P., and Urciuolo, M.}
\newblock About integral operators of fractional type on variable {L}p spaces.
\newblock {\em Georgian Mathematical Journal 20}, 4 (2013), 805--816.

\bibitem{S82}
{\sc Sawyer, E.}
\newblock A characterization of a two-weight norm inequality for maximal
  operators.
\newblock {\em Studia Mathematica 75}, 1 (1982), 1--11.

\end{thebibliography}

\end{document}